\documentclass[12pt]{article}

\usepackage{amsmath,amssymb,amsthm,mathtools,mathrsfs,latexsym}
\usepackage{graphicx}
\usepackage{xcolor}
\usepackage{tikz}
\usepackage[linkcolor=blue,colorlinks=true,urlcolor=red,bookmarksopen=true]{hyperref}
\usepackage{newclude}
\usepackage{imakeidx}
\makeindex

\allowdisplaybreaks

 \numberwithin{equation}{section}

\newtheorem{thm}{Theorem}[section]

\newtheorem{lema}{Lemma}[section]
\newtheorem{prop}{Proposition}[section]

\newtheorem{rmk}{Remark}[section]

\renewcommand{\Pi}{\mathcal{P}}

\newcommand{\Fi}{\mathcal{F}}

\renewcommand{\Xi}{\mathcal{X}}

\newcommand{\R}{\mathbb{R}}

\newcommand{\I}{\mathbb{I}}

\newcommand{\T}{\mathbb{T}}

\let\div\relax
\DeclareMathOperator*{\div}{div}

\DeclareMathOperator*{\Id}{{Id}}

\DeclarePairedDelimiter\abs{\lvert}{\rvert}
\DeclarePairedDelimiter\norm{\lVert}{\rVert}
\makeatletter
\let\oldabs\abs
\def\abs{\@ifstar{\oldabs}{\oldabs*}}
\let\oldnorm\norm
\def\norm{\@ifstar{\oldnorm}{\oldnorm*}}
\makeatother

\usepackage{scalerel}
\usepackage[usestackEOL]{stackengine}
\def\dashint{\,\ThisStyle{\ensurestackMath{
\stackinset{c}{.2\LMpt}{c}{.5\LMpt}{\SavedStyle-}{\SavedStyle\phantom{\int}}
}
\setbox0=\hbox{$\SavedStyle\int\,$}\kern-\wd0}\int}

\newcommand{\br}[1]{\left(#1\right)}

\newcommand{\set}[1]{\left\{#1\right\}}

\usepackage{adjustbox}
\usepackage{wrapfig}
\usepackage{enumitem}
\usepackage{protosem}

\title{Global existence and optimal time-decay rates of the compressible Navier-Stokes equations with density-dependent viscosities}
\date{}
\author{
\bf\large  Jie Fan$^a$, Xiangdi Huang$^a$, Anchun Ni$^a$\\
\small a. State Key Laboratory of Mathematical Sciences, Academy of Mathematics and Systems Science\\
\small Chinese Academy of Sciences, Beijing 100910, P.R.China;\\
\footnote{*Email addresses: fj0828@outlook.com\,\, (J. Fan); xdhuang@amss.ac.cn\,\, (X. Huang);  varnothing@outlook.com\, \, (A. Ni); } }
\begin{document}
\maketitle

\begin{abstract}
This paper  is devoted to studying the Cauchy problem for the three-dimensional isentropic compressible Navier-Stokes equations with density-dependent  viscosities given by $\mu=\rho^\alpha,\lambda=\rho^\alpha(\alpha>0)$. We establish the global existence and  optimal decay rates of classical solutions under the assumptions of small initial data in $L^1(\mathbb{R}^3)\cap L^2(\mathbb{R}^3)$ and the viscosity constraint  $|\alpha-1|\ll 1$. The key idea of our proof lies in the combination of Green's function method, energy method and a time-decay regularity criterion. In contrast to previous works, the Sobolev norms of the spatial derivatives of the initial data may be arbitrarily large in our analysis. \\[4mm]
{\bf Keywords:} compressible Navier-Stokes equations; global well-posedness; optimal decay rates; density-dependent  viscosities\\[4mm]
{\bf Mathematics Subject Classifications (2010):} 76N10, 35B45.\\[4mm]
\end{abstract}

\section{Introduction}
We consider the isentropic compressible Naiver-Stokes equations with  density-dependent viscosities in $\R^3$:
\begin{equation}
\label{CNS}
	\begin{dcases}
		\partial_t \rho +\div (\rho u)=0,\\
		\partial_t(\rho u )+ \div (\rho u\otimes u)+\nabla P=\div \T,\\
		\rho|_{t=0}=\rho_0, u|_{t=0}=u_0,
	\end{dcases}
\end{equation}
Here $\rho(t,x)$ and $u(t,x)$ denote the density and velocities of the fluid respectively. The pressure $P$ and the viscous stress tensor $\T$ are given by:
\begin{equation}
\label{pressure law}
	\begin{split}
		&P=\frac{1}{\gamma}\rho^\gamma  \text{ and }  \T=2\mu (\rho) Du+\lambda(\rho) \div u  \I_3, 
	\end{split}
\end{equation}
where we assume $\gamma>1$ and $Du =\frac{1}{2} (\nabla u+(\nabla u)^T )$ is the deformation tensor and $I_3$ is the three-dimensional identity matrix. The shear and bulk viscosities $ \mu (\rho), \lambda(\rho)$ are dependent with the density and satisfy the following hypothesis:
\begin{equation}
\label{viscosity}
	\mu (\rho)=\mu \rho^\alpha, \lambda(\rho)=\lambda\rho^\alpha \text{ and } \alpha>0, \mu>0, 2\mu +3\lambda\ge 0.
\end{equation}

The investigation of the compressible Navier–Stokes equations with density-dependent viscosities $\mu = \mu(\rho)$ and $\lambda = \lambda(\rho)$ is complicated due to the strong nonlinearity and possible degeneracy. Vasseur-Yu \cite{2016 Vasseur-Yu-IM}, and Li-Xin \cite{2015 Li-Xin-Preprint} independently proved the global existence of weak solutions in the density-degenerate case, based on the significant a priori estimates developed by Bresch-Desjardins \cite{2003 Bresch-Desjardins-CMP}, and Mellet-Vasseur \cite{2007 Mellet-Vasseur-CPDE}. The global well-posedness of classical solutions under sufficiently large initial density was established by Yu \cite{2023 Yu-MMAS} and Huang–Li–Zhang \cite{2024 Huang-Li-Zhang-Preprint1}, using a method adapted from Huang–Li–Xin \cite{2012 Huang-Li-Xin-CPAM}, which was originally developed for small-energy classical solutions in the constant-viscosity case. 

The behavior of solutions near non-vacuum equilibrium for the compressible Navier–Stokes equations was first studied by Matsumura-Nishida \cite{1980 Matsumura-Nishida-JMKU}, where they proved the global existence of classical solutions and derived decay rates under the assumption of small initial data in Sobolev space $H^s$. Zeng \cite{1994 Zeng-CPAM} was the first to adopt Green’s function analysis for the one-dimensional isentropic compressible Navier–Stokes equations. Hoff-Zumbrun \cite{1995 Hoff-Zumbrun-IUMJ, 1997 Hoff-Zumbrun-ZAMP} studied multi-dimensional compressible flows with artificial viscosity and established uniform decay rates for diffusive waves. Furthermore, Liu-Wang \cite{1998 Liu-Wang-CMP} and Li \cite{2005 Li-CMP} carefully investigated the properties of Green’s functions for the isentropic and full Navier–Stokes systems, respectively, and presented optimal time-decay rates in odd dimensions, where the weaker Huygens’ principle appears due to stronger dispersion effects in odd-dimensional spaces. Guo-Wang \cite{2012 Guo-Wang-CPDE} also developed a general framework to study the optimal decay rates of higher-order spatial derivatives for dissipative systems, including the compressible Navier–Stokes equations.
It should be noted that all of the above decay estimates are derived under the assumption of small initial data in Sobolev norms:
\begin{equation*}
	\norm{(\rho_0-\tilde \rho , u )}_{H^s}\ll 1,
\end{equation*}
where $\tilde \rho$ is the equilibrium density. Nevertheless, few studies have investigated the decay rates of solutions when the viscosities are density-dependent.
Luo-Yang \cite{2025-luo-jde} obtained the optimal decay rates for higher-order derivatives in the case where viscosity depends on density ($\mu(\rho)=\alpha\rho,\lambda(\rho)=\beta\rho$), under the following conditions:
$$\sum_{l=0}^{N}\Big(\|\nabla^l(\rho-1)\|^2_{L^2}+\|\nabla^l u\|^2_{L^2}\Big)\ll1.$$

In this paper, our goal is to remove the restriction of requiring small derivatives and to establish the global existence and optimal decay rates of solutions to \eqref{CNS}–\eqref{viscosity} under the following partial smallness condition:
\begin{equation*}
	\norm{(\rho_0-\tilde \rho , u)}_{L^1\cap L^2}\ll 1,
\end{equation*}
and a proper restriction to the parameter $\alpha$.

To formulate our problem, we consider solutions to the compressible Navier–Stokes equations \eqref{CNS}–\eqref{viscosity} with density-dependent viscosities, near the equilibrium state $(\tilde \rho, 0)$. For convenience,  we set $\tilde \rho=1$ and take the replacement:
$$\rho \to \varrho +\tilde \varrho=\varrho +1, u\to u.$$
Then we obtain the following linearized compressible Naiver-Stokes equations:
\begin{equation}
\label{linearized system}
	\begin{dcases}
		\partial_t \varrho  +\nabla\cdot u=N_\varrho (\varrho, u),\\
		\partial_t u+\nabla \varrho-(\mu \Delta u+(\mu+\lambda)\nabla (\nabla \cdot u ) ) =N_u (\varrho, u),\\
        \varrho|_{t=0}=\varrho_0=\rho_0-1, u|_{t=0}=u_0,
	\end{dcases}
\end{equation}
where the nonlinear terms $N_\varrho, N_u$ are given by:
\begin{equation}
\label{nonlinear rho}
	N_\varrho (\varrho, u)=-\nabla\cdot (\varrho u),\quad
		N_u(\varrho, u)= N_1+N_2+N_3+N_4,\\
\end{equation}
with
\begin{equation}
\label{nonlinear u}
	\begin{split}
		N_1(\varrho, u)=& -u\cdot \nabla u,\\
		N_2(\varrho ,u)=& -H_\gamma(\varrho)\nabla \varrho, H_\gamma(\varrho)=(1+\varrho)^{\gamma-2}-1,\\
		N_3(\varrho, u)=& G_\alpha (\varrho) \br{2\mu \nabla \varrho \cdot Du +\lambda \nabla\varrho \div u }, G_\alpha (\varrho) =\alpha (1+\varrho)^{\alpha-2}, \\
		N_4(\varrho ,u)=& H_\alpha (\varrho)\br{\mu\Delta u+(\mu+\lambda )\nabla \div u }, H_\alpha (\varrho)=(1+\varrho)^{\alpha-1}-1.\\
	\end{split}
\end{equation}

The main result of this paper can be stated as follows:
\begin{thm}
\label{main theorem}
	Assume that the initial data $(\varrho_0, u_0)\in H^4(\R^3).$ There exists  sufficiently small constants $C_0,\delta>0$ such that if
	\begin{equation}
	\label{initial smallness}
		\begin{split}
			\norm{(\varrho_0, u_0 ) }_{L^1\cap L^2}\le C_0,
		\end{split}
	\end{equation}
    and\begin{equation}\label{alphasmall}
		|\alpha-1|\le \delta,
	\end{equation}
	then the Cauchy problem \eqref{CNS} has a unique global classical solution in $\mathbb{R}^3\times (0,\infty)$   satisfying
	\begin{equation*}
		(\varrho, u)\in C([0,\infty); H^4), \nabla u\in L^2([0,\infty);H^4).
	\end{equation*}
	Furthermore, we have the following optimal uniform decay estimate:
	\begin{equation}
	\label{optimal decay}
	\begin{split}
		&\norm{(\varrho, u) (t) }_{L^\infty}\le C(1+t)^{-\frac 32},\\
		&\norm{\nabla (\varrho, u) (t) }_{L^\infty}\le C(1+t)^{-2},\\
	\end{split}
	\end{equation}
	and the optimal higher-order decay estimates:
	\begin{equation}
	\label{Hk decay}
		\norm{\nabla^k (\varrho, u) }_{L^2}\le C(1+t)^{-\frac 34-\frac{k}{2}}, k=0,1,2.
	\end{equation}
For $2\leq q\leq\infty,$ it also holds
    \begin{equation}\label{lpoptimal}
\|\nabla^{k'}(\varrho,u)\|_{L^p} \leq C(1+t)^{-\frac{3}{2}(1-\frac{1}{p})-\frac{k'}{2}},
\end{equation}
with $k'=0,1.$
\end{thm}

\begin{rmk}
\label{rmk 1.1}
	It should be noted that $L^1$-smallness is necessary to derive the optimal decay estimates \eqref{optimal decay} and \eqref{Hk decay}, although the solution may not  preserve in $L^1(\mathbb{R}^3)$ due to the possibility of negative perturbation of density.
\end{rmk}
\begin{rmk}
Unlike classical results \cite{2012 Guo-Wang-CPDE, 1995 Hoff-Zumbrun-IUMJ, 1997 Hoff-Zumbrun-ZAMP} that require smallness in the $H^N$-norm (with $N \geq 2$), our framework imposes smallness only on the $L^1\cap L^2$-norm while allowing the derivatives to remain bounded. Our results permits the Sobolev norms of the spatial derivatives of the initial data to be arbitrarily large.
\end{rmk}
\begin{rmk}
The condition \eqref{alphasmall} is essential for closing the high-order estimates, which suggests that our method might fail in the constant viscosity case. Moreover, our analysis covers the case $\alpha=1$.
\end{rmk}

\begin{rmk}
We first obtain the uniform decay estimates \eqref{optimal decay} via the continuity argument, upon which the optimal decay estimates \eqref{Hk decay}-\eqref{lpoptimal} is subsequently established.
\end{rmk}

We now outline the proof for Theorem \ref{main theorem} and highlight the key challenges and techniques involved. Our main strategy combines the linear analysis via the Green function (see Section \ref{sec-2}), and a dedicated nonlinear analysis which couples conditional energy estimates with a regularity criterion, which is formulated as uniform optimal decay estimates (see Section \ref{sec-3}).

 For the linear analysis, we have leveraged the advantages of the Green function of the linearized Navier-Stokes equations, and give out the following decomposition for the Green function $G$ according to the frequency analysis:
\begin{equation*}
	G=G_L+G_{HR}+G_{HS}.
\end{equation*}
Here $G_L$ denotes the low-frequency part, which behaves similarly to the heat kernel. $G_{HR}$ is the regular part related to the high-frequency regime, and it can accommodate one derivative in regularity estimates. $G_{HS}$ represents the singular part of the Green’s function; it acts like a Dirac delta function at the high-frequency level, but fortunately only appears in the continuity equation. The part is standard and can be checked in many other works concerning compressible Naiver-Stokes equation and related models (see \cite{2010 Chen-Miao-Zhang-CPAM, 1995 Hoff-Zumbrun-IUMJ, 1997 Hoff-Zumbrun-ZAMP, 2024 Qi-Wang-DCDS}).

\clearpage

Our main innovation focuses on the proof of nonlinear stability, which is sketched through the following bootstrap process:
\begin{equation*}
	\text{regularity criterion}\Longrightarrow \text{conditional energy estimates}\Longrightarrow\text{improved regularity criterion}.
\end{equation*}
The regularity criterion is designed in accordance with the linear stability and optimal uniform decay estimates (see \eqref{optimal decay}):
\begin{equation*}
	\norm{(\varrho, u) (t) }_{L^\infty}\le C (1+t)^{-\frac 32},\\
		\norm{\nabla (\varrho, u) (t) }_{L^\infty}\le C(1+t)^{-2}.
\end{equation*}

We first attempt to establish conditional energy estimates under the above time-decay regularity criterion, together with the $L^2$-smallness assumption on the initial data \eqref{initial smallness} and the viscosity constraint \eqref{alphasmall}. It is notable that the classical energy estimate is proved by constructing a family of energy functionals coupled with each other, due to the lack of a dissipative estimate for the density as the continuity equation is hyperbolic (see, e.g. Guo-Wang \cite{2012 Guo-Wang-CPDE}). Consequently, it requires the smallness of the initial data in the Sobolev norm ($H^{[\frac N2]+2}$ for $N>3$).

In contrast to the classical approach, our conditional energy estimate needs only $L^2$-smallness of the initial data, with no smallness condition is imposed on the derivatives. This is because the time-decay regularity criterion enables us to close energy inequality at each $l-$th level.  In this process, the following nearly-linear condition 
$$|\alpha - 1| \le \delta(\mu),$$
for the viscosity  $\mu(\rho )=\mu \rho^\alpha, \lambda(\rho)=\lambda \rho^\alpha $
 also plays a significant role in absorbing the higher-order dissipative terms for the velocity, since the linear expansion of $H_{\alpha}(\rho)\propto (\alpha-1)$, which acts as extra coefficients on the dissipation terms.

With the conditional energy estimates at hand, we are able to close the time-decay regularity criterion by the Duhamel formula \eqref{Duhamel}, using the regularity estimates (see Lemma \ref{regularity criterion}) of the Green function as well as the a priori optimal decay estimates. Careful decomposition and interpolation techniques are utilized in this process. Particularly, a $L^{1+}$ axuiliary estimate (where we choose as $L^{\frac 43}$) is introduced to obtain the optimal uniform decay estimate \eqref{optimal decay}, since the $L^1$ estimate is not expected as we addressed in the Remark \ref{rmk 1.1}. It is also noteworthy that to close the a priori assumption \eqref{priori estimate}, the singular part of the Green's function requires the estimates of $\|\nabla^2 (\varrho,u)\|_{L^1L^\infty}$ for the solution, while the second-order nonlinear terms involving $\nabla^2 u$ additionally demand higher regularity of the solution. 

Finally, based on the uniform decay estimates established in Section \ref{sec-3}, we obtain the global existence and optimal decay rates of the solution by employing the standard continuity argument and the Green’s function method in Section \ref{sec-4}.

The rest of the paper is organized as follows.

In Section \ref{sec-2}, we formulate the problem using the Green's function and provide a frequency analysis of its different components. In Section \ref{sec-3}, we establish both basic and higher-order energy estimates, design a bootstrap scheme, and close it using these estimates along with the regularity properties of the Green's function. In Section \ref{sec-4}, we verify the global well-posedness and derive the $H^2$-optimal decay estimates, thereby completing the proof of the main theorem.

\section{Preliminary}
\label{sec-2}

For convenience in the nonlinear analysis, we express system \eqref{linearized system} in operator form:
\begin{equation*}
	\begin{split}
		\partial_t V+LV=N(V).
	\end{split}
\end{equation*}
Here we denote
\begin{equation*}
	V=\begin{pmatrix}
			\varrho \\ u
		\end{pmatrix}, L=\begin{pmatrix}
			0 & \div \\
			\nabla  & -(\mu \Delta +(\mu +\lambda)\nabla \div  )\Id_{n\times n}
		\end{pmatrix}, N(V)=\begin{pmatrix}
			N_\varrho (V)\\
			N_u (V)
		\end{pmatrix}.
\end{equation*}
Clearly, $L$ is a self-adjoint operator and can be characterized by its symbol. According to the Duhamel principle, the solution can be expressed as follows:
\begin{equation}
\label{Duhamel}
\begin{split}
	V(t)=&e^{-tL}V_0 +\int_0^t e^{-(t-s)L} N(V) (s)ds= G(t)*V_0 +\int_0^t G(t-s)*N(V)(s)ds.
\end{split}
\end{equation}
Here we have $ \Fi({ e^{-tL}f})=\hat G(t)\hat f.$

For completeness, we provide a sketch of the analysis of the Green function $G(t,x)$. We aim to compute its Fourier transform $\hat G(t,\xi)$ and derive regularity estimates via frequency analysis. For more details on the linear analysis, please refer to \cite{2010 Chen-Miao-Zhang-CPAM, 1995 Hoff-Zumbrun-IUMJ, 1997 Hoff-Zumbrun-ZAMP, 2024 Qi-Wang-DCDS}.

\begin{enumerate}[wide, labelwidth=!, labelindent=0pt]
\item \textbf{ Explicit form of Green function.} 
To analyze the density component of the Green function, a key observation is that we can derive the following wave-type equation from the linear part of \eqref{linearized system}:
\begin{equation*}
	\varrho_{tt}-\nabla \varrho -(2\mu+\lambda)\Delta \varrho_t=0.
\end{equation*}
In the frequency form, it can be written as the following second-order ODE:
\begin{equation}
\label{ODE}
	\frac d{dt} \begin{pmatrix}
		\hat \varrho \\ \hat \varrho_t
	\end{pmatrix}
	=\begin{pmatrix}
		0 & 1\\
		-|\xi |^2 & -(2\mu+\lambda ) |\xi |^2
	\end{pmatrix}
	\begin{pmatrix}
		\hat \varrho \\ \hat \varrho_t
	\end{pmatrix}.
\end{equation}
The the general solution of \eqref{ODE} is
\begin{equation*}
	\begin{split}
		\hat \varrho(t)=A(\xi )e^{\lambda_-(\xi)t }+B(\xi)e^{\lambda_+(\xi)t},
	\end{split}
\end{equation*}
where the eigenvalues $\lambda_\pm$ is determined by the following characteristic equation
\begin{equation*}
\begin{split}
	&\lambda^2+(2\mu+\lambda ) |\xi |^2\lambda+ |\xi|^2=0.\\
	\end{split}
\end{equation*}
We can explicitly solve for the eigenvalues as
\begin{equation*}
	\lambda_{\pm} (\xi)=\frac{-\nu |\xi|^2\pm \sqrt{\nu^2|\xi|^4-4|\xi|^2 } }{2},
\end{equation*}
where $\nu =(2\mu+\lambda)$. And the coefficients $A(\xi), B(\xi)$ are determined by the initial data:
\begin{equation*}
	\begin{dcases}
		A(\xi)+B(\xi)=\hat \varrho_0,\\
		\lambda_-A(\xi)+\lambda_+(\xi)B(\xi)=-i\xi  \hat u_0,
	\end{dcases}
\end{equation*}
which can be solved as
\begin{equation*}
\label{coefficients}
	\begin{dcases}
		A(\xi )=\frac{\lambda_+\hat \varrho+i\xi \hat u_0 }{\lambda_+-\lambda_- },\\
		B(\xi )=-\frac{\lambda_-\hat \varrho+i\xi \hat u_0 }{\lambda_+-\lambda_- }.
	\end{dcases}
\end{equation*}
By combining the explicit form of coefficients and eigenvalues,  we obtain the following expression for the density:
\begin{equation}
\label{expression rho}
\begin{split}
	\hat \varrho(t)
	=& \frac{\lambda_+e^{\lambda_-t }-\lambda_-e^{\lambda_+t } }{\lambda_+-\lambda_- }\hat \varrho_0+i\xi \frac{e^{\lambda_-t }-e^{\lambda_+t } }{\lambda_+-\lambda_- }\hat u_0.
\end{split}
\end{equation}
Moreover, after the Fourier transform, the linear part of momentum equation becomes
\begin{equation}
\label{equation u}
	 \hat u_t +i \xi \hat \varrho+ (\mu|\xi |^2 I+(\mu+\lambda )\xi  \xi^T  )\hat u=0.
\end{equation}
Since it is a multi-dimensional ODE, we decompose the momentum equation orthogonally with respect to the frequency direction. Denote $\xi=|\xi|\xi'$, then we have
\begin{equation*}
	\hat u =(\hat u\cdot \xi' )\xi'+(\hat u-(\hat u\cdot \xi' )\xi' )\coloneqq a\xi'+b, a\xi'\parallel \xi, b\perp \xi. 
\end{equation*}
The first and third terms are parallel to $\hat{u}$, which justifies the use of the above decomposition. The second and fourth terms are parallel to $\xi$.
\begin{equation*}
\begin{split}
	i\xi \varrho= i|\xi |\varrho \xi', 	\quad (\mu+\lambda )\xi\xi^T \hat u
	 = (\mu +\lambda )|\xi|^2a\xi'.
\end{split}
\end{equation*}
As a result, the multi-dimensional ODE can be decomposed into two scalar ODEs:
\begin{equation*}
	\begin{dcases}
		\tilde \varrho a_t +(2\mu+\lambda )|\xi |^2 a+i|\xi |\hat \varrho=0,\\
		\tilde \varrho b_t +\mu  |\xi |^2b=0,\\
		a_0=\hat u_0\cdot \xi', b_0=\br{I-\frac{\xi\xi^T }{|\xi |^2 }}\hat u_0 .
	\end{dcases}
\end{equation*}
Using the above results and standard ODE techniques, we can derive that
\begin{equation*}
	\begin{dcases}
	a=-i |\xi|\frac{e^{\lambda_+ t }-e^{\lambda_-t } }{\lambda_+-\lambda_- }  \hat \varrho_0+\frac{\lambda_+e^{\lambda_+ t }-\lambda_-e^{\lambda_-t }  }{\lambda_+-\lambda_- } (u_0\cdot \xi' ),\\
		b=e^{-\mu|\xi |^2t/\tilde \varrho}\br{I-\frac{\xi\xi^T }{|\xi |^2 }}\hat u_0.
	\end{dcases}
\end{equation*}
The solution of momentum equation \eqref{equation u} can then be expressed as
\begin{equation}
\label{expression u}
	\hat u=-i \frac{e^{\lambda_+ t }-e^{\lambda_-t } }{\lambda_+-\lambda_- }  \hat \varrho_0\xi+\br{e^{-\mu |\xi|^2 t }I-\br{\frac{\lambda_+e^{\lambda_+ t }-\lambda_-e^{\lambda_-t }  }{\lambda_+-\lambda_- }-e^{-\mu|\xi|^2t }}\xi'\xi'^T }\hat u_0.  
\end{equation}

Summarizing the above results \eqref{expression rho} and \eqref{expression u}, we derive the explicit form of the Green function in the frequency level.
\begin{lema}
	Suppose the Green function $G(t,x)$ of the linearized Navier–Stokes equations \eqref{linearized system} is given by formula \eqref{Duhamel}. Then its Fourier transform $\hat{G}(t,\xi)$ can be expressed as follows:
	\begin{equation}
    \label{Green function}
		\begin{split}
			\hat G(t,\xi)=\begin{pmatrix}
			\frac{\lambda_+e^{\lambda_- t}-\lambda_-e^{\lambda^+ t} }{\lambda_+-\lambda_- } & -i  \frac{e^{\lambda_+ t}-e^{\lambda_- t} }{\lambda_+-\lambda_- } \xi^T\\
			-i  \frac{e^{\lambda_+ t}-e^{ \lambda_- t} }{\lambda_+-\lambda_- } \xi &\ \ \  e^{-\mu|\xi|^2t }(I-\xi' \xi'^T)+\frac{\lambda_+e^{\lambda_- t}-\lambda_-e^{\lambda_+ t} }{\lambda_+-\lambda_- } \xi'\xi'^T
		\end{pmatrix},
		\end{split}
	\end{equation}
	where the eigenvalues $\lambda_{\pm}$ is defined by
	\begin{equation}
    \label{eigenvalues}
		\lambda_{\pm} (\xi)=\frac{-\nu |\xi|^2\pm \sqrt{\nu^2|\xi|^4-4|\xi|^2 } }{2}, \nu =(2\mu+\lambda).\\
	\end{equation}
\end{lema}
\item \textbf{Frequency analysis of Green function.} To fully exploit the stability mechanism of the linear operator, we consider the low-frequency and high-frequency parts separately.

If $|\xi| \ll 1$, we expand the Green function around $|\xi| = 0$. Then the eigenvalues $\lambda_{\pm}$ given by \eqref{eigenvalues} have the following asymptotic expansion:
\begin{equation*}
	\begin{split}
		\lambda_+(\xi)=i|\xi|-\frac{\nu}{2}|\xi |^2 +O (|\xi|^3 ), \lambda_-(\xi)=-i|\xi |-\frac{\nu }{2}|\xi |^2+O( |\xi |^3 );\\
		e^{\lambda_+t}(\xi)=e^{i|\xi |t }e^{-\frac{\nu }{2}|\xi |^2t } e^{O(|\xi |^3 )t }, e^{\lambda_+t}(\xi)=e^{-i|\xi |t }e^{-\frac{\nu }{2}|\xi |^2t } e^{O(|\xi |^3 )t };\\
		\lambda_+-\lambda_-(\xi)=2i|\xi |+O(|\xi|^3 ), \frac{1}{\lambda_+-\lambda_- }(\xi)=-\frac i 2|\xi |^{-1}+O(|\xi | );\\
		\frac{\lambda_+ }{\lambda_+-\lambda_- }(\xi)=\frac 12 +i O( |\xi | ), \frac{\lambda_- }{\lambda_+-\lambda_- }(\xi)=-\frac 12 +i O( |\xi | ).
	\end{split}
\end{equation*}
Utilizing the above asymptotic analysis and explicit form \eqref{Green function} of Green function $G(t,x)$, we find that it behaves like the heat kernel in the low-frequency regime.

If $|\xi| \gg 1$, we now expand the Green function around $|\xi|^{-1} \to 0$. Then
\begin{equation*}
	\begin{split}
		\lambda_+(\xi)=-1+O(|\xi|^{-2}), \lambda_-(\xi)=-|\xi|^{-2}+1+O(|\xi|^{-2});\\
		\frac{\lambda_+(\xi)}{\lambda_+(\xi)-\lambda_-(\xi) } e^{\lambda_+(\xi)t }=e^{-t}O(|\xi|^{-2});\\
		 \frac{\lambda_-(\xi)}{\lambda_+(\xi)-\lambda_-(\xi)}e^{\lambda_+(\xi)t}=e^{-t}(-1+O(|\xi|^{-2}));\\
		\frac{1}{\lambda_+(\xi)-\lambda_-(\xi) }e^{\lambda_+(\xi)t}=e^{-t}O(|\xi|^{-2}).
	\end{split}
\end{equation*}
We can see that the singularity occurs in the first term of the Green function:
\begin{equation*}
	\begin{split}
		\frac{\lambda_+(\xi)e^{\lambda_-(\xi) t}-\lambda_-(\xi)e^{\lambda^+(\xi) t} }{\lambda_+(\xi)-\lambda_-(\xi) }\sim  {e^{-c t}}-|\xi |^{-2}e^{-c' |\xi |^2 t }.
	\end{split}
\end{equation*}
Inspiring by the above analysis, we take decomposition as follows:
\begin{equation*}
	G= G_L+G_{HR}+G_{HS}.
\end{equation*}
Here $G_L$ denotes the low frequency part of the Green function, defined by
\begin{equation}
\label{GL}
	G_L=\chi (D)G,\quad \chi(\xi)\coloneqq \begin{dcases}
		1, & |\xi|\le \frac 12,\\
		0, & |\xi|>1,
	\end{dcases}
\end{equation}
and $G_{HR}$ and $G_{HS}$ denote the regular and singular parts of the Green function in the high-frequency regime:
\begin{equation}
\label{GR GS}
	G_{HS}=e^{-ct}\delta(x)\begin{pmatrix}
		1-\chi(D) & 0 & 0\\
		0 & 0 & 0 \\
		0 & 0 & 0
	\end{pmatrix},\quad
	G_{HR}=(1-\chi (D))G-G_{HS}.
\end{equation}
We have following standard estimates for the Green function. More details of the proof can be found in  \cite{2010 Chen-Miao-Zhang-CPAM, 1997 Hoff-Zumbrun-ZAMP, 2024 Qi-Wang-DCDS}. 

\end{enumerate}

\begin{lema}
\label{regularity criterion}
	Suppose $G(x,t)$ is the Green function of the linearized Navier–Stokes equations \eqref{linearized system}, and that $G_L$, $G_{HR}$, and $G_{HS}$ are defined by \eqref{GL} and \eqref{GR GS}, respectively. 
	\begin{enumerate}
		\item For any $ k\ge 0 $ and $ p\in [1,\infty]$, we have:
		\begin{equation}
\label{low frequency}
	\begin{split}
		\norm{D^k G_L(\cdot,t) }_{L^p}\le & C(1+t)^{-\frac{3}{2}(1-\frac 1p )-\frac k2}.\\
	\end{split}
\end{equation}
	\item For any $k=0,1 $ and $ p\in [1,\infty], $ we have:
	\begin{equation}
\label{regular high}
		\norm{D^k G_{HR} (\cdot,t)}_{L^p}\le Ce^{-Ct}.
\end{equation}
\item For any $p\in [1,\infty]$ and $f\in L^p(\R^3)$, we have:
\begin{equation}
\label{singular}
	\norm{G_{HS}*f }_{L^p}\le Ce^{-Ct} \norm{f}_{L^p}.
\end{equation}
	\end{enumerate}
\end{lema}
\begin{rmk}
	Compared to the singular part, $G_{HR}$ can bear a first derivative.
\end{rmk}

\section{ A priori estimates}\label{sec-3}
This section establishes uniform a priori estimates for the system \eqref{CNS}. Fix $T > 0$, and let $(\varrho, u)$ be a smooth solution on $\mathbb{R}^3 \times (0, T]$ with initial data $(\varrho_0, u_0)$ satisfying the prescribed regularity conditions. Below, we derive essential a priori estimates for this solution.
\begin{prop}\label{pr11}
		Assume that the conditions of Theorem \ref{main theorem} hold. There exists a positive constant $\eta$ depending on $\lambda,\mu,\gamma$ such that if $(\varrho,u)$ is a smooth solution to the system \eqref{CNS} on $\mathbb{R}^3 \times (0,T]$ satisfying 
\begin{equation}\label{priori assumption}
\begin{gathered}
\|(\varrho, u)\|_{L^\infty} \leq \frac{1}{2}(1+t)^{-\frac{3}{2}}, \quad \|\nabla(\varrho, u)\|_{L^\infty} \leq \eta(1+t)^{-2},
\end{gathered}
\end{equation}
then the following estimates hold:
\begin{equation}\label{priori estimate}
\begin{gathered}
\|(\varrho, u)\|_{L^\infty} \leq \frac{1}{4}(1+t)^{-\frac{3}{2}}, \quad \|\nabla(\varrho, u)\|_{L^\infty} \leq \frac{\eta}{2}(1+t)^{-2}.
\end{gathered}
\end{equation}
	\end{prop}
In the following discussion, we will first establish conditional energy estimates, and then use them to close the regularity criterion.

\subsection{Conditional energy estimate}

In this subsection, under the assumptions of time-decaying regularity criteria and small initial data in $L^2$, we derive uniform energy estimates for the solution and its derivatives. To ensure the global existence of smooth solutions, it suffices to establish the following a priori estimates.

	The following lemma is necessary to handle the nonlinear terms involving the density:
	\begin{lema}
Assume that \(\|\varrho\|_{L^\infty} \leq 1\). Suppose $\varrho \in H^k(\R^3)$ and $f: \R\to \R$ is a smooth function with bounded derivatives up to order $k\geq1$, then for any $2\leq p\leq\infty$, the following estimate holds:
		\begin{equation}
		\label{density function}
			\|\nabla^kf(\varrho)\|_{L^p}\le C\norm{\nabla^k \varrho }_{L^p},
		\end{equation}
    where the constant $C>0$ depends on $\|f\|_{C^k}, p$ and $k.$
	\end{lema}
	\begin{proof}
    Notice that for \( k \geq 1 \),

\[
\nabla^k (f(\varrho)) = \text{a sum of products } g^{\gamma_1, \ldots, \gamma_n} (\varrho) \nabla^{\gamma_1} \varrho \cdots \nabla^{\gamma_n} \varrho,
\]
where the functions \( f^{\gamma_1, \ldots, \gamma_n} (\varrho) \) are some derivatives of \( f(\varrho) \) and \( 1 \leq \gamma_i \leq k \), \( i = 1, \ldots, n \) with \( \gamma_1 + \cdots + \gamma_n = k \). We then use the Sobolev interpolation inequality and H\"older inequality to obtain
		\begin{equation}\begin{aligned}  
\|\nabla^kf(\varrho)\|_{L^p}\leq&C\|\nabla^{\gamma_1}\varrho\cdot\cdot\cdot\nabla^{\gamma_n}\varrho\|_{L^p}\\
\leq&C\|\nabla^{\gamma_1}\varrho\|_{L^{p_1}}\cdot\cdot\cdot\|\nabla^{\gamma_n}\varrho\|_{L^{p_j}}\\
\leq&C\|\nabla^k\varrho\|_{L^p}^{\frac{p}{pk-3}\Big((\gamma_1+\cdot\cdot\cdot+\gamma_j)-3(\frac{1}{p_1}+\cdot\cdot\cdot+\frac{1}{p_j})\Big)}\|\varrho\|^{\zeta}_{L^\infty}\\
\leq&C\|\nabla^k\varrho\|_{L^p}\|\varrho\|^{\zeta}_{L^\infty},
\end{aligned}\end{equation} 
where $\frac{1}{p_1}+\cdot\cdot\cdot+\frac{1}{p_j}=\frac{1}{p}$ and $\zeta>0$ is a positive constant.
Consequently, the lemma holds provided that $\|\varrho\|_{L^\infty} \leq 1$.
\end{proof}
We also state the classical commutator estimate below, which will be used in the higher-order energy estimates.
\begin{lema}\cite{{A. Adams},{majda}}
Let $k \geq 1$ be an integer, then we have
\begin{equation}\begin{aligned}\label{iii}
\|\nabla^k(gh)\|_{L^r(\mathbb{R}^3)} \leq C\|g\|_{L^{r_1}(\mathbb{R}^3)}\|\nabla^kh\|_{L^{r_2}(\mathbb{R}^3)}+C\|h\|_{L^{r_3}(\mathbb{R}^3)}\|\nabla^kg\|_{L^{r_4}(\mathbb{R}^3)},
\end{aligned}\end{equation}
with $1 < r_1, r_2,r_3, r_4 < \infty$ and $r_i$ $(1 \leq i \leq 4)$ satisfy the following identity:
\[
\frac{1}{r_1}+\frac{1}{r_2}=\frac{1}{r_3}+\frac{1}{r_4}=\frac{1}{r}.
\]
Let $k\geq1$ be an integer and define the commutator
\begin{equation}\begin{aligned}\label{2iii}
[\nabla^k,f]g\overset{\text{def}}{=}\nabla^k(fg)-f\nabla^kg,
\end{aligned}\end{equation}
Then we have
\begin{equation}\begin{aligned}\label{1iii}
\|[\nabla^k,f]g\|_{L^r(\mathbb{R}^3)} \leq C\|\nabla g\|_{L^{r_1}(\mathbb{R}^3)}\|\nabla^{k-1}h\|_{L^{r_2}(\mathbb{R}^3)}+C\|h\|_{L^{r_3}(\mathbb{R}^3)}\|\nabla^kg\|_{L^{r_4}(\mathbb{R}^3)},
\end{aligned}\end{equation}
with $1 < r_1, r_2,r_3, r_4 < \infty$ and $r_i$ $(1 \leq i \leq 4)$ satisfy the following identity:
\[
\frac{1}{r_1}+\frac{1}{r_2}=\frac{1}{r_3}+\frac{1}{r_4}=\frac{1}{r}.
\]
\end{lema}

A critical limitation in Guo-Wang’s analysis \cite{2012 Guo-Wang-CPDE} arises from the lack of density dissipation in equations \eqref{CNS}, which prevents the closure of energy estimates at each order. In contrast to such classical results, our approach exploits a time-decay regularity criterion (see \eqref{priori assumption}), enabling a unified framework where estimates are closed via Gronwall's inequality.

We begin by establishing the basic energy estimate, for which only $L^2$-smallness of the initial data is required—a significant relaxation compared to \cite{2012 Guo-Wang-CPDE}.
\begin{lema}\label{basicestimate}
		Under the assumption  \eqref{priori assumption},
		 there holds that
		\begin{equation}
		\label{basic energy estimate}
		\begin{aligned}
				\sup_{0\leq t\leq T}\Big(\|\varrho\|^2_{L^2}+\|u\|^2_{L^2}\Big)+\int_{0}^{T}\|\nabla u\|^2_{L^2}dt\leq CC_0,
		\end{aligned}\end{equation}
		where $C$ denotes generic positive constant depending only on $a,\lambda,\mu, \alpha,\gamma$.
	\end{lema}
    \begin{proof}
  By multiplying the equation \eqref{linearized system} by $\varrho$ and $u$, respectively, and adding the resulting equations, we obtain  
    \begin{equation}\begin{aligned}\label{i}
        \frac{1}{2}\frac{d}{dt}&\int_{\mathbb{R}^3}\Big(|\varrho|^2+|u|^2\Big)dx+\mu\int_{\mathbb{R}^3}|\nabla u|^2dx+(\lambda+\mu)\int_{\mathbb{R}^3}|\text{div}u|^2dx\\
        =&-\int_{\mathbb{R}^3}(\varrho\text{div u}+u\cdot{\nabla\varrho})\varrho dx\\
        &-\Big(u\cdot{\nabla u}+H_{\gamma}(\varrho)\nabla\varrho-G_{\alpha}(\varrho)(2\mu\nabla\varrho\cdot{Du}+\lambda\nabla\varrho\text{div}u\mathbb{I}_3)\\
        &-H_\alpha(\varrho)(\mu\triangle u+(\lambda+\mu)\nabla\text{div}u)\Big)\cdot{u}dx\\
        =&\sum_{i=1}^{6}I_i.
        \end{aligned}\end{equation}
By combining \eqref {priori assumption} with H\"older's inequality, we derive
\begin{equation}\begin{aligned}\label{i1}
I_1=&-\int_{\mathbb{R}^3}\Big(\varrho\text{div u}+u\cdot{\nabla\varrho}\Big)\varrho dx\\
&\leq C\|\nabla u\|_{L^\infty}\|\varrho\|^2_{L^2}.
 \end{aligned}\end{equation}
 It follows that
   \begin{equation}\begin{aligned}
   I_2\leq C\|\nabla u\|_{L^\infty}\|u\|^2_{L^2},
   \end{aligned}\end{equation}
and
\begin{equation}\begin{aligned}
I_3\leq C\|\nabla\varrho\|_{L^\infty}(\|u\|^2_{L^2}+\|\varrho\|^2_{L^2}).
 \end{aligned}\end{equation}
 By H\"older inequality, we have
 \begin{equation}\begin{aligned}
   I_4\leq& \frac{\mu}{4}\|\nabla u\|^2_{L^2}+C\Big(\|\nabla\varrho\|^2_{L^\infty}+\|\nabla\varrho\|^2_{L^\infty}\|\varrho\|^2_{L^\infty}\Big)\|u\|^2_{L^2}.
   \end{aligned}\end{equation}
   An application of integration by parts yields
   \begin{equation}\begin{aligned}\label{i5}
   I_5=&\int_{\mathbb{R}^3}\Big(H_\alpha(\varrho)(\mu\triangle u+(\lambda+\mu)\nabla\text{div}u\Big)\cdot{u}dx\\
   =&-\int_{\mathbb{R}^3}\nabla(H_\alpha(\varrho)(\mu\nabla u+(\lambda+\mu)\text{div}u)\cdot{u}dx\\
   &-\int_{\mathbb{R}^3}H_\alpha(\varrho)(\mu\nabla u+(\lambda+\mu)\text{div}u)\nabla udx\\
   \leq& \frac{\mu}{4}\|\nabla u\|^2_{L^2}+C\Big(\|\nabla\varrho\|^2_{L^\infty}+\|\nabla u\|^2_{L^\infty}\Big)\Big(\|u\|^2_{L^2}+\|\varrho\|^2_{L^2}\Big).
   \end{aligned}\end{equation}   
   Inserting \eqref{i1}-\eqref{i5} into \eqref{i}, we derive that
 \begin{equation}\begin{aligned}
  \frac{1}{2}\frac{d}{dt}&\int_{\mathbb{R}^3}\Big(|\varrho|^2+|u|^2)dx+\frac{\mu}{2}\int_{\mathbb{R}^3}|\nabla u|^2dx+(\lambda+\mu)\int_{\mathbb{R}^3}|\text{div}u|^2dx\\
\leq&C\Big(\|\nabla u\|_{L^\infty}+\|\nabla\varrho\|_{L^\infty}+\|\nabla\varrho\|^2_{L^\infty}+\|\nabla\varrho\|^2_{L^\infty}\|\varrho\|^2_{L^\infty}+\|\nabla u\|^2_{L^\infty}\Big)\Big(\|u\|^2_{L^2}+\|\varrho\|^2_{L^2}\Big).
   \end{aligned}\end{equation}
The proof is completed by applying Gronwall’s inequality and the  assumption \eqref{priori assumption}.
 \end{proof}

 When closing the high-order energy estimates, higher-order dissipation terms of the velocity will emerge. The introduction of the constraint on the viscosity,
 $$|\alpha-1|\le \delta,$$
is indispensable for absorbing certain dissipation terms.
\begin{lema}\label{higherestimate}
		Under the assumption \eqref{priori assumption}, for any integer $1\leq k\leq4$,
		 there holds that
		\begin{equation}
		\label{high energy estimate}
		\begin{aligned}
				\sup_{0\leq t\leq T}(\|\nabla^k\varrho\|^2_{L^2}+\|\nabla^ku\|^2_{L^2})+\int_{0}^{T}\|\nabla ^{k+1}u\|^2_{L^2}dt\leq C,
		\end{aligned}\end{equation}
		with $C=C(a,\lambda,\mu, \alpha,\gamma,C_0)$.
	\end{lema}
    \begin{proof}
For each integer $k$ satisfying $1\leq k\leq 4$, we first take the $k-$th order derivative of both equations $\eqref{linearized system}_1$ and $\eqref{linearized system}_2$, then multiply the resulting equation by $\nabla^k\varrho$ and $\nabla^k u$ respectively, finally combine  and integrate the sum equations over $\mathbb{R}^3$ to get
\begin{equation}\begin{aligned}\label{L}
  \frac{1}{2}\frac{d}{dt}&\int_{\mathbb{R}^3}(|\nabla^k\varrho|^2+|\nabla^k u|^2)dx+\mu\int_{\mathbb{R}^3}|\nabla^{k+1} u|^2dx+(\lambda+\mu)\int_{\mathbb{R}^3}|\nabla^k\text{div}u|^2dx\\
=&-\int_{\mathbb{R}^3}\nabla^k(\text{div}(\varrho u))\cdot{\nabla^{k}\varrho}dx-\int_{\mathbb{R}^3}\nabla^{k}\Big(u\cdot{\nabla u}+H_{\gamma}(\varrho)\nabla\varrho\\
&-G_\alpha (\varrho)(2\mu\nabla\varrho\cdot{Du}+\lambda\nabla\varrho\text{div}u\mathbb{I}_3)\\
&-H_\alpha (\varrho)(\mu\triangle u-(\lambda+\mu)\nabla\text{div}u)
\Big)\cdot{\nabla^{k}u}dx\\
=&-\int_{\mathbb{R}^3}\nabla^k(\text{div}(\varrho u))\nabla^k\varrho dx+\int_{\mathbb{R}^3}\nabla^{k-1}\Big(u\cdot{\nabla u}+H_{\gamma}(\varrho)\nabla\varrho\\
&-G_\alpha (\varrho)(2\mu\nabla\varrho\cdot{Du}+\lambda\nabla\varrho\text{div}u\mathbb{I}_3)\\
&-H_\alpha (\varrho)(\mu\triangle u-(\lambda+\mu)\nabla\text{div}u)\Big)\cdot{\nabla^{k+1}u}dx\\
=&\sum_{i=1}^{5}L_i.
   \end{aligned}\end{equation} 
$L_1$ can be represented as:
\begin{equation}\begin{aligned}
L_1=&-\int_{\mathbb{R}^3}\nabla^k(\varrho\text{div}u+u\cdot{\nabla\varrho})\cdot{\nabla^{k}\varrho}dx\\
=&L_{11}+L_{12}.
    \end{aligned}\end{equation} 
Using \eqref{iii} and Holder's inequality yields
\begin{equation}\begin{aligned}\label{L1}
L_{11}=&-\int_{\mathbb{R}^3}\nabla^k(\varrho\text{div}u)\cdot{\nabla^{k}\varrho}dx\\
\leq&C\|\varrho\|_{L^\infty}\|\nabla^{k}\div u\|_{L^2}\|\nabla^k\varrho\|_{L^2}+C\|\text{div}u\|_{L^\infty}\|\nabla^k\varrho\|^2_{L^2}\\
\leq&\frac{\mu}{16}\|\nabla^{k+1}u\|^2_{L^2}+C\Big(\|\varrho\|^2_{L^\infty}+\|\nabla u\|_{L^\infty}\Big)\|\nabla^k\varrho\|^2_{L^2}.
    \end{aligned}\end{equation}  
After applying integration by parts and making use of the Sobolev inequality together with \eqref{iii}, \eqref{1iii}, we obtain
\begin{equation}\begin{aligned}\label{L2}
L_{12}=&-\int_{\mathbb{R}^3}\nabla^k(u\cdot{\nabla\varrho})\cdot{\nabla^{k}\varrho}dx\\
\leq&-\int_{\mathbb{R}^3}\text{div}u||\nabla^k\varrho|^2dx+\int_{\mathbb{R}^3}([\nabla^k,u]\cdot{\nabla\varrho})\nabla^k\varrho dx\\
\leq&C\|\nabla u\|_{L^\infty}\|\nabla^k\varrho\|^2_{L^2}+C\|\nabla\varrho\|_{L^\infty}\|\nabla^ku\|_{L^2}\|\nabla^k\varrho\|_{L^2}\\
\leq&C(\|\nabla u\|_{L^\infty}+\|\nabla\varrho\|_{L^\infty})(\|\nabla^ku\|^2_{L^2}+\|\nabla^k\varrho\|^2_{L^2}).
\end{aligned}\end{equation}      
The combination of \eqref{L1} and \eqref{L2} gives
\begin{equation}\begin{aligned}
L_1\leq&\frac{\mu}{16}\|\nabla^{k+1}u\|^2_{L^2}+C\Big(\|\varrho\|^2_{L^\infty}+\|\nabla u\|_{L^\infty}+\|\nabla\varrho\|_{L^\infty}\Big)\Big(\|\nabla^ku\|^2_{L^2}+\|\nabla^k\varrho\|^2_{L^2}\Big).
\end{aligned}\end{equation} 
Similarly, applying \eqref{iii} yields
\begin{equation}\begin{aligned}
L_2=&-\int_{\mathbb{R}^3}\nabla^{k-1}\Big(u\cdot{\nabla u}\Big)\cdot{\nabla^{k+1}u}dx\\
\leq&\frac{\mu}{16}\|\nabla^{k+1}u\|^2_{L^2}+C\|\nabla^{k-1}(u\cdot{\nabla u})\|^2_{L^2}\\
\leq&\frac{\mu}{16}\|\nabla^{k+1}u\|^2_{L^2}+C\|u\|^2_{L^\infty}\|\nabla^ku\|^2_{L^2}+C\|\nabla u\|^2_{L^\infty}\|\nabla^{k-1}u\|^2_{L^2},
\end{aligned}\end{equation} 
and
\begin{equation}\begin{aligned}
L_3=&\int_{\mathbb{R}^3}\nabla^{k-1}(H_\gamma(\varrho))\nabla\varrho)\cdot{\nabla^{k+1}u}dx\\
\leq&C\|\nabla\varrho\|_{L^\infty}\|\nabla^{k-1}H_\gamma(\varrho)\|_{L^2}\|\nabla^{k+1}u\|_{L^2}+C\|H_\gamma(\varrho)\|_{L^\infty}\|\nabla^k\varrho\|_{L^2}\|\nabla^{k+1}u\|_{L^2}\\
\leq&C\Big(\|\nabla\varrho\|_{L^\infty}\|\nabla^{k-1}\varrho\|_{L^2}+\|\varrho\|_{L^\infty}\|\nabla^k\varrho\|_{L^2}\Big)\|\nabla^{k+1}u\|_{L^2}\\
\leq&\frac{\mu}{16}\|\nabla^{k+1}u\|^2_{L^2}+C\|\nabla\varrho\|^2_{L^\infty}\|\nabla^{k-1}\varrho\|^2_{L^2}+C\|\varrho\|^2_{L^\infty}\|\nabla^k\varrho\|^2_{L^2}.
\end{aligned}\end{equation}
Applying Holder's inequality together with \eqref{iii}, we obtain
\begin{equation}\begin{aligned} \label{L5} 
L_4=&\int_{\mathbb{R}^3}\nabla^{k-1}\Big(G_\alpha (\varrho)(2\mu\nabla\varrho\cdot{Du}+\lambda\nabla\varrho\text{div}u\mathbb{I}_3)\Big)\cdot{\nabla^{k+1}u}dx\\
\leq&C\Big(\|G_\alpha (\varrho)\|_{L^\infty}\|\nabla^{k-1}(\nabla\varrho\nabla u)\|_{L^2}\\
&+\|\nabla\varrho\|_{L^\infty}\|\nabla u\|_{L^\infty}\|\nabla^{k-1}(G_\alpha (\varrho))\|_{L^2}
\Big)\|\nabla^{k+1}u\|_{L^2}\\
\leq&\frac{\mu}{16}\|\nabla^{k+1}u\|^2_{L^2}+\Big(1+\|\varrho\|^2_{L^\infty}\Big)\Big(\|\nabla\varrho\|^2_{L^\infty}\|\nabla^{k}u\|^2_{L^2}+\|\nabla u\|^2_{L^\infty}\|\nabla^{k}\varrho\|^2_{L^2}\Big)\\
&+\|\nabla \varrho\|^2_{L^\infty}\|\nabla u\|^2_{L^\infty}\|\nabla^{k-1}\varrho\|^2_{L^2}\\
\leq&\frac{\mu}{16}\|\nabla^{k+1}u\|^2_{L^2}+\Big(\|\nabla\varrho\|^2_{L^\infty}+\|\nabla u\|^2_{L^\infty}\|\varrho\|^2_{L^\infty}+\|\nabla u\|^2_{L^\infty}\\
&+\|\nabla \varrho\|^2_{L^\infty}\|\varrho\|^2_{L^\infty}\Big)\Big(\|\nabla^{k}u\|^2_{L^2}+\|\nabla^{k}\varrho\|^2_{L^2}\Big)\\
&+\|\nabla\varrho\|^2_{L^\infty}\|\nabla u\|^2_{L^\infty}\|\nabla^{k-1}\varrho\|^2_{L^2}.
\end{aligned}\end{equation}
As for $L_5$, we first establish through Taylor expansion 
\begin{equation*}
	\|H_\alpha (\varrho)\|_{L^\infty}\leq C|\alpha-1|\|\varrho\|_{L^\infty}.
\end{equation*} 
 A careful case-by-case analysis is required for the estimation of $L_5$. Specifically:

When $k=1, p=2,q=\infty,$  we observe
\begin{equation}\begin{aligned}
L_5=&\int_{\mathbb{R}^3}\Big(H_\alpha (\varrho)(\mu\triangle u+(\lambda+\mu)\nabla\text{div}u)\Big)\cdot{\nabla^{2}u}dx\\
\leq&C\|H_\alpha (\varrho)\|_{L^\infty}\|\nabla^{2}u\|^2_{L^2}\\
\leq&C|\alpha-1|\|\varrho\|_{L^\infty}\|\nabla^{2}u\|^2_{L^2}\\
\leq&\frac{\mu}{16}\|\nabla^{2}u\|^2_{L^2}.
\end{aligned}\end{equation}
where we have use \eqref{alphasmall}.

 In the case where $k\geq2,$ it follows from \eqref{alphasmall} and \eqref{iii} that
\begin{equation}\begin{aligned}
L_5\leq&\frac{\mu}{32}\|\nabla^{k+1}u\|^2_{L^2}+C\|\nabla^{k-1}(H_\alpha (\varrho)\nabla^2u)\|^2_{L^2}\\
\leq&(\frac{\mu}{32}+\|H_\alpha (\varrho)\|^2_{L^\infty})\|\nabla^{k+1}u\|^2_{L^2}+C\|\nabla^{k-1}H_\alpha (\varrho)\|^2_{L^4}\|\nabla^2 u\|^2_{L^4}\\
\leq&(\frac{\mu}{32}+\|H_\alpha (\varrho)\|^2_{L^\infty})\|\nabla^{k+1}u\|^2_{L^2}+C\|\nabla^{k-1}\varrho\|^2_{L^4}\|\nabla^2 u\|^2_{L^4}\\
\leq&(\frac{\mu}{32}+C|\alpha-1|^2\|\varrho\|^2_{L^\infty})\|\nabla^{k+1}u\|^2_{L^2}+C\|\varrho\|^{\frac{1}{2k-3}}_{L^\infty}\|\nabla u\|^{\frac{4k-7}{2k-3}}_{L^\infty}\|\nabla^k\varrho\|^{\frac{4k-7}{2k-3}}_{L^2}\|\nabla^{k+1}u\|^{\frac{1}{2k-3}}_{L^2}\\
\leq&\frac{\mu}{16}\|\nabla^{k+1}u\|^2_{L^2}+C\|\varrho\|^{\frac{2}{4k-7}}_{L^\infty}\|\nabla u\|^{2}_{L^\infty}\|\nabla^k\varrho\|^{2}_{L^2}.
\end{aligned}\end{equation}

By collecting all the estimates from \eqref{L1} to \eqref{L5} and returning to equation \eqref{L}, we establish that

\begin{equation}\begin{aligned}\label{ll}
\frac{1}{2}\frac{d}{dt}&\int_{\mathbb{R}^3}(|\nabla^k\varrho|^2+|\nabla^k u|^2)dx+\mu\int_{\mathbb{R}^3}|\nabla^{k+1} u|^2dx+(\lambda+\mu)\int_{\mathbb{R}^3}|\nabla^k\text{div}u|^2dx\\
\leq&\frac{5\mu}{16}\|\nabla^{k+1}u\|^2_{L^2}+\Big(\|\nabla\varrho\|_{L^\infty}+\|\nabla u\|_{L^\infty}+\|\varrho\|^2_{L^\infty}+\|u\|^2_{L^\infty}\\
&+\|\varrho\|^{\frac{2}{4k-7}}_{L^\infty}\|\nabla u\|^{2}_{L^\infty}+\|\nabla\varrho\|^2_{L^\infty}+\|\nabla u\|^2_{L^\infty}\|\varrho\|^2_{L^\infty}+\|\nabla u\|^2_{L^\infty}\\
&+\|\nabla \varrho\|^2_{L^\infty}\|\varrho\|^2_{L^\infty}\Big)\Big(\|\nabla^{k}u\|^2_{L^2}+\|\nabla^{k}\varrho\|^2_{L^2}\Big)\\
&+\Big(\|\nabla\varrho\|^2_{L^\infty}\|\nabla u\|^2_{L^\infty}+\|\nabla\varrho\|^2_{L^\infty}+\|\nabla u\|^2_{L^\infty}\Big)\Big(\|\nabla^{k-1}\varrho\|^2_{L^2}+\|\nabla^{k-1}u\|^2_{L^2}\Big).\\
\end{aligned}\end{equation}
Combining \eqref{ll} with Gronwall's inequality, we obtain the desired estimate \eqref{high energy estimate}.
    \end{proof}
    
\subsection{Bootstrap}\label{Bootstrap}
\label{Section 4}
	In this section, our goal is to close the a priori estimates \eqref{priori assumption}. According to the Duhamel formula \eqref{Duhamel}, we shall consider the linear part $G(t)*V_0$ and the nonlinear part $\int_0^t G(t-s)*N(V)(s)ds$ respectively. Furthermore, we recall that
	\begin{equation*}
	G=G_{L}+G_{HR}+G_{HS},
	\end{equation*}
	 and  the nonlinear term:
	\begin{equation*}
		N(V)=
        \begin{pmatrix}
        N_\varrho(V)    \\
        N_u (V)
        \end{pmatrix}
        , \quad N_u(V) =N_1(V)+N_2(V)+N_3(V)+N_4(V),
	\end{equation*}
	where $N_\varrho$ and $N_i(i\le 4)$  are defined as \eqref{nonlinear rho} and \eqref{nonlinear u}. Recall \eqref{priori assumption}, notice the regularity criterion can be rewritten as
\begin{equation}
\label{regularity criteria}
	\norm{V(t)}_{L^\infty}\le \frac{1}{2} (1+t)^{-\frac 32}, \norm{\nabla V(t) }_{L^\infty}\le \eta(1+t)^{-2},
\end{equation}
as well as the initial data:
\begin{equation}
\label{initial V}
	\norm{ V_0}_{L^1\cap L^2}\le C_0.
\end{equation} 
Then by the regularity criterion \eqref{regularity criteria}, we have that
	\begin{equation}
	\label{observation}
		|N_\varrho (\varrho ,u)|+|N_1 (\varrho ,u)|+|N_2 (\varrho ,u) |\le C\abs{V}\abs{\nabla V }.
	\end{equation}
In the following discussion, we will use the above observations to simplify the calculations.

In the subsequent closure of the a priori estimates for the regularity criteria, the estimate for $\|N(V)\|_{L^1}$ plays an important role. This is because obtaining control of $\|V\|_{L^1}$ is inherently challenging. Thus, it is necessary to first establish a estimate of $\|V\|_{L^{\frac{4}{3}}}$, which is crucial for closing the time-decay regularity criterion.
\begin{lema}
	Suppose regularity criteria \eqref{regularity criteria} and the initial assumption \eqref{initial V} are valid. Then we have
	\begin{equation}
	\label{V 43}
		\norm{V(t) }_{L^{\frac 43}}\le CC_0(1+t)^{-\frac{7}{40}},\,\,t\in[0,T].
	\end{equation}
\end{lema}
\begin{proof}Applying the Duhamel formula \eqref{Duhamel}, we have
\begin{equation}
	\label{linear 4}
	\begin{split}
 \|V(t)\|_{L^{\frac{4}{3}}} \leq &\|G(t) * V_0 \|_{L^{\frac{4}{3}}}+ \int_0^t \|G(t - s) * N(V)(s)\|_{L^{\frac{4}{3}}} ds\\
 \leq&\|G(t) * V_0 \|_{L^{\frac{4}{3}}}+\int_0^t \|G_L(t - s) * N(V)(s)\|_{L^{\frac{4}{3}}} ds\\
 &+\int_0^t \|G_{HR}(t - s) * N(V)(s)\|_{L^{\frac{4}{3}}} ds+\int_0^t \|G_{HS}(t - s) * N(V)(s)\|_{L^{\frac{4}{3}}} ds\\
 =:&\sum_{i=1}^{4}K_j.
 \end{split}
\end{equation}
Here $K_1$ is the linear part, and $K_2, K_3,K_4$ correspond to the nonlinear part.
 
 By the decay estimate \eqref{low frequency}-\eqref{singular} of the Green function and $L^1-$smallness \eqref{initial V} of the initial data, it comes that
	\begin{equation}
	\label{linear 4.1}
	\begin{split}
		K_1\le & \norm{G_L(t) *V_0}_{L^\frac{4}{3}}+\norm{G_{HR}(t)*V_0}_{L^\frac{4}{3}}+\norm{G_{HS}(t)*V_0}_{L^\frac{4}{3}}\\
		\le & \norm{G_L(t)}_{L^\frac{4}{3}}\norm{V_0}_{L^1}+\norm{G_{HR}(t)}_{L^\frac{4}{3}}\norm{V_0}_{L^1}+e^{-Ct}\norm{V_0}_{L^\frac{4}{3}}\\
		\le & C(1+t)^{-\frac 38}C_0+Ce^{-Ct}C_0+Ce^{-Ct}C_0\le C(1+t)^{-\frac 38}C_0.
	\end{split}
\end{equation}
As for the nonlinear term, we can derive the estimate for the low-frequency part $\int_0^t G_L(t-s)*N(V)(s)ds$ by using \eqref{basic energy estimate} and \eqref{observation}:
\begin{equation*}
	\begin{split}
		K_2\leq&\int_0^{t}\| G_L(t-s)*(N_\varrho, \  N_1+N_2 +N_3+N_4)^T(s)\|_{{L^\frac 43}}ds\\
       \leq&C\int_{0}^{t}(1+t-s)^{-\frac{3}{2}(1-\frac{3}{4})}\Big(\|\text{div}(\varrho u)\|_{L^1}+\|u\cdot{\nabla}u\|_{L^1}\\
&+\|H_\gamma (\varrho)\nabla\varrho\|_{L^1}+\|G_\alpha (\varrho)\nabla\varrho\nabla u\|_{L^1}+\|H_\alpha (\varrho)\nabla^2u\|_{L^1}\Big)\\
\leq&C\int_{0}^{t}(1+t-s)^{-\frac{3}{2}(1-\frac{3}{4})}\Big((\|u\|_{L^2}+\|\varrho\|_{L^2})(\|\nabla u\|_{L^2}+\|\nabla\varrho\|_{L^2})\\
&+(1+\|\varrho\|_{L^{\infty}})(\|\nabla u\|_{L^2}\|\nabla\varrho\|_{L^2})+\|\varrho\|_{L^2}\|\nabla^2u\|_{L^2}+(1+\|\varrho\|_{L^{\infty}})(\|\nabla u\|_{L^2}\|\nabla\varrho\|_{L^2})\Big)\\
\leq&C\int_{0}^{t}(1+t-s)^{-\frac{3}{8}}(\|V\|_{L^2}\|\nabla V\|_{L^2}++\norm{\nabla V }_{L^2}^2+\|V\|_{L^2}\|\nabla^2 V\|_{L^2}+\norm{\nabla V }_{L^2}^2)\\
\leq&C\int_{0}^{t}(1+t-s)^{-\frac{3}{8}}(C_0^{\frac{8}{5}}(1+s)^{-\frac{4}{5}}+C_0(1+s)^{-1}+C_0^{\frac{6}{5}}(1+s)^{-\frac{8}{5}})\\
\leq&CC_0(1+t)^{-\frac{7}{40}},
	\end{split}
\end{equation*}
where we use the following Gargliardo-Nirenberg interpolations by energy estimates \eqref{basic energy estimate}, \eqref{high energy estimate} and the regularity criteria \eqref{regularity criteria}:
\begin{equation}
\label{grad V 2}
	\begin{split}
		\norm{\nabla V(t) }_{L^2}\le  C\norm{V(t)}_{L^2}^{\frac{3}{5}}\norm{\nabla V(t) }_{L^\infty}^{\frac 2{5}}\le C C_0^\frac{3}{5}(1+t)^{-\frac{4}{5}},\\
	\end{split}
\end{equation}
\begin{equation}
\label{grad2 V 2}
	\norm{\nabla^2 V (t)}_{L^2}\le C\norm{V(t) }_{L^\infty}^\frac{2}{3}\norm{\nabla^3 V(t) }_{L^2}^\frac 13 \le C(1+t)^{-1}.
\end{equation}Consequently, we can obtain that
\begin{equation}
\label{low 4.1}
	K_2\leq\int_0^t \norm{G_L(t-s)*N(V)(s)ds }_{L^\frac 43}\le CC_0(1+t)^{-\frac{7}{40}}.
\end{equation}
Since $G_{HR}$ is decay faster than $G_L$ in $L^p$ sense by \eqref{regular high}, it comes naturally:
\begin{equation}
\label{regular 4.1}
	K_3\leq\int_0^t \norm{G_{HR}(t-s)*N(V)(s)ds }_{L^\frac 43}\le CC_0(1+t)^{-\frac{7}{40}}.
\end{equation}
We finally handle the singular part using \eqref{singular}:
\begin{equation}
\label{singular 4.1}
	\begin{split}
		&K_4\leq\int_0^t \norm{G_{HS}(t-s)*N_\varrho(\varrho ,u)(s)ds}_{L^\frac 43}\le C\int_0^t e^{-C(t-s)}\norm{\text{div} (\varrho u) (s)}_{L^\frac 43}ds\\
		&\le C\int_0^t e^{-C(t-s)}\Big((\|u\|_{L^2}+\|\varrho\|_{L^2})(\|\nabla u\|_{L^4}+\|\nabla\varrho\|_{L^4})\Big)ds\\
		&\le  CC_0 \int_0^t e^{-C(t-s)}(1+s)^{-\frac{7}{5}}ds\le CC_0(1+t)^{-\frac{7}{40}} .\\
	\end{split}
\end{equation}
Here we apply the following interpolation in the fourth step by \eqref{basic energy estimate} and \eqref{regularity criteria}:
\begin{equation}
\label{grad V 4}
\norm{\nabla V(t) }_{L^4 }\le C\norm{V}_{L^2}^\frac 3{10} \norm{ \nabla V }_{L^\infty}^\frac{7}{10}\le CC_0^{\frac 3{10}} (1+t)^{-\frac 75}.
\end{equation}
Combining the above estimates \eqref{linear 4.1}-\eqref{singular 4.1} together, the proof is finished.
\end{proof}
With this $L^{1+}$ estimate at hand, we can now derive a refined uniform decay estimate for $V(t)=(\varrho,u)(t)$ by assuming $C_0$ small enough.
\begin{lema}
\label{lema 4.2}
	Suppose that regularity criteria \eqref{regularity criteria} and the initial assumption \eqref{initial V} is valid. Then there holds
	\begin{equation}
	\label{V infty}
		\norm{V(t) }_{L^\infty}\le \frac{1}{4} (1+t)^{-\frac 32},\,\,\,\,t\in[0,T].
	\end{equation}
\end{lema}
\begin{proof}
We use the Duhamel formula \eqref{Duhamel} to get that
\begin{equation}
	\label{linear 5}
	\begin{split}
 \|V(t)\|_{L^{\infty}} \leq &\|G(t) * V_0 \|_{L^{^{\infty}}}+ \int_0^t \|G(t - s) * N(V)(s)\|_{L^{^{\infty}}} ds\\
 \leq&\|G(t) * V_0 \|_{L^{\infty}}+\int_0^t \|G_L(t - s) * N(V)(s)\|_{L^{\infty}} ds\\
 &+\int_0^t \|G_{HR}(t - s) * N(V)(s)\|_{L^{^{\infty}}} ds+\int_0^t \|G_{HS}(t - s) * N(V)(s)\|_{L^{^{\infty}}} ds\\
 =:&\sum_{i=1}^{4}M_j
 \end{split}
\end{equation}
Then we have the uniform estimate for the linear part:
	\begin{equation}
	\label{M1}
	\begin{split}
		M_1\le & \norm{G_L(t) *V_0}_{L^\infty}+\norm{G_{HR}(t)*V_0}_{L^\infty}+\norm{G_{HS}(t)*V_0}_{L^\infty}\\
		\le & \norm{G(t)}_{L^\infty}\norm{V_0}_{L^1}+\norm{G_{HR}}_{L^\infty}\norm{V_0}_{L^1}+e^{-Ct}\norm{V_0}_{L^\infty}\\
        \le & \norm{G(t)}_{L^\infty}\norm{V_0}_{L^1}+\norm{G_{HR}}_{L^\infty}\norm{V_0}_{L^1}+e^{-Ct}\norm{V_0}_{L^2}^\frac 14 \norm{\nabla^2 V_0 }_{L^2}^\frac 34\\
		\le & C(1+t)^{-\frac 32}C_0+Ce^{-Ct}C_0+Ce^{-ct}C_0^\frac 14\le CC_0^\frac 14 (1+t)^{-\frac 32}.
	\end{split}
\end{equation}
Furthermore, for the estimate of the low frequency part, we use \eqref{low frequency}, \eqref{basic energy estimate}, \eqref{high energy estimate}, \eqref{regularity criteria}, \eqref{grad V 4} and above $L^\frac 43$-estimate \eqref{V 43} to get:
\begin{equation*}
	\begin{split}
		M_2\leq&\int_0^{t} \norm{G_L(t-s)*(|N_\varrho| +|N_1|+|N_2|+|N_3| )}_{L^\infty }ds\\
        &+ \int_0^t \norm{G_L(t-s)*N_4 (\varrho, u)(s)}_{L^\infty }ds\\
        =:&M_{21}+M_{22}.
        \end{split}
\end{equation*}
Applying \eqref{V 43}, \eqref{grad V 4} and \eqref{grad V 2}, we arrive at
\begin{equation*}
	\begin{split}
		M_{21}\leq&\int_{0}^{\frac{t}{2}}\norm{G_L(t-s)}_{L^\infty}\| (|N_\varrho| +|N_1|+|N_2|+|N_3|)(s)\|_{L^1}ds\\
        &+\int_{\frac{t}{2}}^{t}\norm{G_L(t-s)}_{L^2}\|(|N_\varrho| +|N_1|+|N_2|+|N_3|)(s)\|_{L^2}ds\\
        \leq&\int_{0}^{\frac{t}{2}}(1+t-s)^{-\frac{3}{2}}\Big(\|\text{div}(\varrho u)\|_{L^1}+\|u\cdot{\nabla}u\|_{L^1}\\
        &+\|H_\gamma(\varrho)\nabla\varrho\|_{L^1}+\|G_{\alpha}(\varrho)\nabla\varrho\nabla u\|_{L^1}\Big)ds\\
&+\int_{\frac{t}{2}}^{t}(1+t-s)^{-\frac{3}{4}}\Big(\|\text{div}(\varrho u)\|_{L^2}+\|u\cdot{\nabla}u\|_{L^2}\\
&+\|H_\gamma(\varrho)\nabla\varrho\|_{L^2}+\|G_\alpha(\varrho)\nabla\varrho\nabla u\|_{L^2}\Big)ds\\
\leq&\int_{0}^{\frac{t}{2}}(1+t-s)^{-\frac{3}{2}}\Big((\|u\|_{L^{\frac{4}{3}}}+\|\varrho\|_{L^{\frac{4}{3}}})(\|\nabla\varrho\|_{L^4}+\|\nabla u\|_{L^4})\\
&+(1+\|\varrho\|_{L^\infty})\|\nabla\varrho\|_{L^2}\|\nabla u\|_{L^2}\Big)ds \\
&+\int_{\frac{t}{2}}^{t}(1+t-s)^{-\frac{3}{4}}\Big((\|\varrho\|_{L^2}+\|u\|_{L^2})(\|\nabla\varrho\|_{L^\infty}+\|\nabla u\|_{L^\infty})\\
&+(1+\|\varrho\|_{L^\infty})\|\nabla\varrho\|_{L^2}\|\nabla u\|_{L^\infty}\Big)ds\\
        \le & \int_0^\frac t2 (1+t-s)^{-\frac{3}{2}}(\norm{V(s)}_{L^{\frac 43}}\norm{\nabla V(s)}_{L^4} +\norm{\nabla V }^2_{L^2})ds\\
		& + \int_\frac t2 ^t (1+t-s)^{-\frac{3}{4}} (\norm{V(s)}_{L^2}\norm{\nabla V (s)}_{L^\infty}+\norm{\nabla V }_{L^2}\norm{\nabla V }_{L^\infty}) ds \\
		\le & C_0^{\frac{6}{5}} \int_0^\frac t2 (1+t-s)^{-\frac 32}(1+s)^{- \frac{7}{5}-\frac{7}{40}}ds+ C_0\int_\frac t2^t (1+t-s)^{-\frac 34}(1+s)^{- 2}ds \\
		\le& CC_0^{\frac{6}{5}} (1+t)^{-\frac 32},
	\end{split}
\end{equation*}
Before addressing $M_{22}$, we employ the following interpolation using the regularity criteria  \eqref{regularity criteria} and the highest-order energy estimate  \eqref{high energy estimate} with $k=4$, neither of which provides smallness
\begin{equation}\label{nabla2v}
	\norm{\nabla^2 V(t) }_{L^4}\le C\norm{\nabla V }_{L^\infty}^\frac 56 \norm{\nabla^4 V }_{L^2}^\frac 16\le C(1+t)^{-\frac 53}.
\end{equation}
Therefore, applying \eqref{V 43} and \eqref{nabla2v}, we have
\begin{eqnarray}
		M_{22}&=&\int_0^t \norm{G_L(t-s)*N_4 (\varrho, u)(s)}_{L^\infty }ds\nonumber\\
		&\le & C\int_0^t \norm{G_L(t-s) }_{L^\infty }\norm{\varrho\nabla^2 u (s)}_{L^1}ds\\
		&\le &C\int_0^t \norm{G_L(t-s)}_{L^\infty }\norm{V(s)}_{L^\frac 43}\norm{\nabla^2 V(s) }_{L^4}ds\nonumber\\
		&\le & CC_0 \int_0^t (1+t-s)^{-\frac 32}(1+s)^{-\frac 53-\frac {7}{40}}\le CC_0 (1+t)^{-\frac 32}.\nonumber
	\end{eqnarray}
Summing the above terms, we can obtain:
\begin{equation}
\label{M2}
	M_2\leq\int_0^t \norm{G_L(t-s)*N(V)(s)}_{L^\infty}ds \le C C_0^{\frac{3}{5}} (1+t)^{-\frac 32}.
\end{equation}
Since we only use the estimate for $\norm{G(t)}_{L^p}$, by \eqref{regular high} it comes similarly:
\begin{equation}
\label{M3}
	M_3\leq\int_0^t \norm{G_{HR}(t-s)*N(V)(s)}_{L^\infty}ds \le C C_0 (1+t)^{-\frac 32}.
\end{equation}
We now attempt to derive an estimate for the singular part. The following refined uniform estimate is crucial for gaining the required smallness:
\begin{equation}
\label{V infty estimate}
	\norm{V(t)}_{L^\infty}\le C\norm{V(t)}_{L^2}^{\frac{2}{5}}\norm{\nabla V(t) }_{L^\infty}^{\frac 3{5}}\le C C_0^{\frac{2}{5}}(1+t)^{-\frac{6}{5}}, 
\end{equation}
 Using \eqref{singular} and \eqref{regularity criteria} and the above estimate \eqref{V infty estimate}, we can derive that
\begin{equation}
\label{M4}
	\begin{split}
		M_4\leq&\int_0^t \norm{G_{HS}(t-s)*N_\varrho (\varrho, u)(s)}_{L^\infty}ds\\
		\le & C\int_0^t e^{-C(t-s)}\norm{V(s) }_{L^\infty}\norm{\nabla V (s)}_{L^\infty}ds\\
		\le & CC_0^{\frac{2}{5}}\int_0^t e^{-C(t-s)}(1+s)^{-\frac{6}{5}}(1+s)^{-2}ds\le CC_0^{\frac{2}{5}} (1+t)^{-\frac 32} .\\
	\end{split}
\end{equation}
Substitute \eqref{M1}, \eqref{M2}, \eqref{M3}, and \eqref{M4} into \eqref{linear 5}, then we can get  the refined uniform decay estimates \eqref{V infty} finally.
\end{proof}
\begin{lema}
\label{lema 4.3}
Suppose that regularity criteria \eqref{regularity criteria} and the initial assumption \eqref{initial V} are valid. Then we have
\begin{equation}
\label{gradient V infty}
	\norm{\nabla V(t) }_{L^\infty}\le \frac{\delta}{2} (1+t)^{-2},\,\,t\in[0,T].
\end{equation}	
\end{lema}
\begin{proof}
We use again the Duhamel formula \eqref{Duhamel} to get
\begin{equation}
	\label{linear 6}
	\begin{split}
 \|\nabla V(t)\|_{L^{\infty}} \leq &\|\nabla(G(t) * V_0) \|_{L^{^{\infty}}}+ \int_0^t \|\nabla(G(t - s) * N(V)(s))\|_{L^{^{\infty}}} ds\\
 \leq&\|\nabla(G(t) * V_0) \|_{L^\infty}+\int_0^t \|\nabla(G_L(t - s) * N(V)(s))\|_{L^\infty} ds\\
 &+\int_0^t \|\nabla(G_{HR}(t - s) * N(V)(s))\|_{L^\infty} ds+\int_0^t \|\nabla(G_{HS}(t - s) * N(V)(s))\|_{L^\infty} ds\\
 =:&\sum_{i=1}^{4}N_j.
 \end{split}
\end{equation}
Then we can use \eqref{low frequency}, \eqref{regular high}, and \eqref{singular} to derive the gradient estimate for the linear part:
	\begin{equation}
	\label{N1}
		\begin{split}
			N_1\le & \norm{\nabla G_L* V_0 }_{L^\infty}+\norm{\nabla G_{HR}* V_0 }_{L^\infty}+\norm{\nabla G_{HS}* V_0 }_{L^\infty}\\
			\le & C\norm{V_0}_{L^1} { \norm{\nabla G_{L} }_{L^\infty}}+C\norm{\nabla V_0 }_{L^2}\norm{G_{HR}}_{L^2}+C\norm{\nabla V_0 }_{L^\infty}e^{-Ct}\\
            \leq&\norm{V_0}_{L^1} { \norm{\nabla G_{L} }_{L^\infty}}+C\norm{V_0}_{L^2}^\frac{1}{2}\norm{\nabla^2 V_0 }_{L^2}^{\frac 1{2}}\norm{G_{HR}}_{L^2}+C\norm{V_0}_{L^2}^\frac 16\norm{\nabla^3 V_0 }_{L^2}^{\frac 56}\|G_{HS}\|_{L^1}\\
			\le & CC_0(1+ t)^{-2 }+CC_0^{\frac{1}{2}} e^{-Ct}+CC_0^{\frac{1}{6}} e^{-Ct}\\
            \le & CC_0^{\frac{1}{6} }(1+t)^{-2}
            \leq  \frac{\eta}{4}(1+t)^{-2}.
		\end{split}
	\end{equation}
	We now handle the nonlinear term. For the low-frequency part, we consider the following decomposition: 
		\begin{equation*}
	\begin{split}
		N_{2}\leq&\int_0^t\norm{\nabla G_L (t-s)*(|N_\varrho| +|N_1| +|N_2|+|N_3| ) (s)}_{L^\infty}ds+\int_0^t\norm{\nabla G_L (t-s)*(N_4) (s)}_{L^\infty}ds \\
        =:&N_{11}+N_{12}.
        \end{split}
\end{equation*}
Here we use \eqref{low frequency}, \eqref{regularity criteria} and interpolations \eqref{grad V 2}, \eqref{grad V 4} and \eqref{V infty} to get
\begin{eqnarray}
		N_{21}&\leq&\int_0^\frac t2\norm{\nabla G_L(t-s) }_{L^\infty }\Big(\|\text{div}(\varrho u)\|_{L^1}+\|u\cdot{\nabla}u\|_{L^1}\nonumber\\
        &&+\|H_\gamma (\varrho)\nabla\varrho\|_{L^1}+\|G_\alpha(\varrho)\nabla\varrho\nabla u\|_{L^1}\Big)ds\nonumber\\
&&+\int_{\frac{t}{2}}^{t}\norm{\nabla G_L(t-s) }_{L^1 }\Big(\|\text{div}(\varrho u)\|_{L^\infty}+\|u\cdot{\nabla}u\|_{L^\infty}\nonumber\\
&&+\|H_\gamma(\varrho)\nabla\varrho\|_{L^\infty}+\|G_\alpha(\varrho)\nabla\varrho\nabla u\|_{L^\infty}\Big)ds\\
&\leq&\int_{0}^{\frac{t}{2}}(1+t-s)^{-2}\Big((\|u\|_{L^{\frac{4}{3}}}+\|\varrho\|_{L^{\frac{4}{3}}})(\|\nabla\varrho\|_{L^4}+\|\nabla u\|_{L^4})\nonumber\\
&&+(1+\|\varrho\|_{L^\infty})\|\nabla\varrho\|_{L^2}\|\nabla u\|_{L^2}\Big)ds\nonumber\\
&&+\int_{\frac{t}{2}}^{t}(1+t-s)^{-\frac{1}{2}}\Big((\|u\|_{L^\infty}+\|\varrho\|_{L^\infty})(\|\nabla\varrho\|_{L^\infty}+\|\nabla u\|_{L^\infty})\nonumber\\
&&+(1+\|\varrho\|_{L^\infty})\|\nabla\varrho\|_{L^\infty}\|\nabla u\|_{L^\infty}\Big)ds\nonumber\\
        &\le & \int_0^\frac t2 (1+t-s)^{-2}\Big(\norm{V (s)}_{L^\frac 43}\norm{\nabla V }_{L^4}+\norm{\nabla V(s) }_{L^2}^2\Big) ds\nonumber\\
        &&+\int_\frac t2^t (1+t-s)^{-\frac{1}{2}}\Big(\norm{V(s)}_{L^\infty}\norm{\nabla V(s) }_{L^\infty}+\norm{\nabla V(s) }_{L^{\infty}}^2 \Big)ds\nonumber\\
		&\le & CC_0^{\frac{6}{5}} \int_0^\frac t2 (1+t-s)^{-2}(1+s)^{-\frac 75-\frac{7}{40}}ds+CC_0^{\frac{1}{12}} \int^t_\frac t2 (1+t-s)^{-\frac 12 }(1+s)^{-2-1}ds\nonumber\\
        &\le &  CC_0^ {\frac{1}{12} }(1+t)^{-2},\nonumber
	\end{eqnarray}
where in the lasts step we use the following interpolation to gain the smallness:
\begin{equation}
\label{grad V infty estimate}
	\begin{split}
\|\nabla V\|_{L^\infty}\leq&C\|\nabla^3u\|^{\frac{2}{3}}_{L^2}\|u\|^{\frac{1}{3}}_{L^\infty}
\leq CC_0^{\frac{1}{12}}(1+t)^{-1}.
\end{split}
\end{equation}
and
\begin{equation*}
	\begin{split}
		N_{22}\leq&\int_0^t\norm{\nabla G_L (t-s)*N_4(\varrho, u)(s) }_{L^\infty}ds\\
        \le & \int_0^\frac t2 \norm{\nabla  G_L(t-s) }_{L^\infty }\norm{\varrho\nabla^2 u (s)}_{L^1}ds+\int_\frac t2^t \norm{\nabla G_{L}(t-s) }_{L^{\frac{4}{3}}}\norm{\varrho \nabla^2 u (s)}_{L^4}ds \\
		\le &C\int_0^\frac t2 \norm{\nabla G_L(t-s)}_{L^\infty }\norm{V(s)}_{L^\frac 43 }\norm{ \nabla^2 V(s) }_{L^4}ds\\
		&+C\int_ {\frac{t}{2}}^{t} \norm{\nabla G_L(t-s)}_{L^\frac 43 }\norm{V(s)}_{L^\infty }\norm{\nabla^2 V(s) }_{L^4 }ds \\
		\le & CC_0 \int_ {0}^{\frac{t}{2}}(1+t-s)^{-2}(1+s)^{-\frac 53-\frac{7}{40}}ds+CC_0^{\frac{1}{4}}\int_{\frac{t}{2}}^{t}(1+t-s)^{-\frac78}(1+s)^{-\frac 32-\frac 53}ds\\
        \le& CC_0^{\frac{1}{12}} (1+t)^{-2}.
	\end{split}
\end{equation*}
Combining the above two estimates together, we can obtain
\begin{equation}
\label{N2}
	\begin{split}
		 N_2\leq\int_0^t \norm{\nabla(G_L(t-s) *N(V)(s))}_{L^\infty}ds\le CC_0^{\frac{1}{12}} (1+t)^{-2}.
	\end{split}
\end{equation}
As for the part involving $G_{HR}$, from \eqref{regular high} we have
\begin{equation*}
	\begin{split}
		\norm{\nabla G_{HR}(t) }_{L^p}\le Ce^{-Ct}.
	\end{split}
\end{equation*}
Then it comes similarly:
\begin{equation}
\label{N3}
	\begin{split}
		N_3\leq \int_0^t \norm{\nabla (G_{HR}(t-s) *N(V)(s))}_{L^\infty}ds\le CC_0^{\frac{1}{12}} (1+t)^{-2}.
	\end{split}
\end{equation}
We now turn to the singular part, which requires the following estimates
\begin{equation}
\label{nabla2infty}
\|\nabla^2V\|_{L^\infty}\leq C\|\nabla^4V\|_{L^2}\|\nabla V\|_{L^\infty}
\leq C(1+t)^{-\frac{2}{3}}.
\end{equation}
Using the estimate for the Green function \eqref{singular} and interpolations \eqref{V infty estimate} and \eqref{grad V infty estimate}, \eqref{nabla2infty}, we can get
\begin{equation}
\label{N4}
\begin{split}
	N_4\leq&\int_0^t \norm{\nabla(G_{HS} (t-s)*N_\varrho (\varrho ,u) (s))}_{L^\infty}ds \le \int_0^t\norm{G_{HS}(t-s)* \nabla^2(\varrho u)(s) }_{L^\infty}ds\\
	\le & C\int_0^t e^{-C(t-s)}\br{\norm{V (s)}_{L^\infty}\norm{\nabla^2V (s)}_{L^\infty}+ \norm{\nabla V(s) }_{L^\infty}^2 }ds\\
	\le & CC_0^{\frac{1}{12}} \int_0^t e^{-C(t-s)}\br{(1+s)^{-\frac 32-\frac 23}+(1+s)^{-1-2} }ds\le CC_0^{\frac{1}{12}}(1+t)^{-2}.
\end{split}
\end{equation}
Summing over \eqref{N1}, \eqref{N2}, \eqref{N3}  \eqref{N4}, we obtain
\begin{equation}
\label{Nall}
\begin{split}
\int_0^t \|\nabla(G(t-s)*N(V)(s))\|_{L^\infty}ds
\leq&CC_0^{\frac{1}{12}}(1+t)^{-2}
\leq \frac{\eta}{4}(1+t)^{-2}.
\end{split}
\end{equation}
Combining \eqref{Nall} and \eqref{N1} yields \eqref{gradient V infty}.
\end{proof}
\begin{proof}[The proof of the Proposition \ref{pr11}]
Using Lemma \ref{lema 4.2} and Lemma \ref{lema 4.3} directly, the proof is completed.
\end{proof}
\section{Proof of Theorem \ref{main theorem}}
\label{sec-4}
With all the  conditional energy estimates and the a priori estimates at hand, we are ready to construct the global classical solutions emerged from  $L^1\cap L^2$ initial data and its optimal decay rates. As the local well-posedness have been verified by Zhang-Zhao \cite{2014 zhang-zhao}, it remained to establish the global well-posedness of classical solutions. 
\subsection{Global existence of classic solutions}

We now adopt the continuity argument to close the discussion. Set
\begin{equation*}\label{4.2}
	T^*=\sup\set{T| \eqref{priori assumption} \text{ holds on } [0,T] }.
\end{equation*}
We claim that 
\begin{equation}
		T^*=\infty.
	\end{equation}
    To prove this, assume by contradiction that $T^*<\infty$. By virtue of Proposition \ref{pr11} and Lemma \ref{basicestimate}-\ref{higherestimate}, the solution $(\rho,u)|_{t=T^*}$ satisfies the regularity conditions of the initial values. Applying local existence of solutions to this terminal data allow us to extend the classical solution to $\mathbb{R}^3\times (0, T_{**}]$ with $T^{**}>T^*,$ contradicting the maximality of $T^*$ in \eqref{4.2}. Therefore, the assumption $T^*<\infty$ must be false, and the solution must exist globally, i.e., $T^*=\infty.$

\subsection{Optimal time-decay estimates in $H^k$-norm ($0\leq k\le 2$)} 
It remains to check the $H^k$-decay estimates \eqref{Hk decay} for $k=0,1,2$. The difficulties comes mainly from the high-order derivatives, as $G_{HR}$ contains at most first-order derivatives. We follows the strategy in Section \ref{Section 4} to derive the decay rate. 

For the case $k=0$, using the Duhamel formula \eqref{Duhamel} and 
\begin{eqnarray}
 \| V(t)\|_{L^2} &\leq &\|G(t) * V_0 \|_{L^2}+ \int_0^t \|G(t - s) * N(V)(s)\|_{L^2} ds\nonumber\\
 &\leq&\|G_L(t) * V_0 \|_{L^2}+\|G_{HR}(t) * V_0 \|_{L^2}+\|G_{HS}(t) * V_0 \|_{L^2}\nonumber\\
 &&+\int_0^t \|G_L(t - s) * N(V)(s)\|_{L^2} ds+\int_0^t \|G_{HR}(t - s) * N(V)(s)\|_{L^2} ds\nonumber\\
 &&+\int_0^t \|G_{HS}(t - s) * N(V)(s)\|_{L^2} ds\nonumber\\
&\leq&C\|G_L(t)\|_{L^2}\| V_0\|_{L^1} +C\|G_{HR}(t)\|_{L^2}\| V_0\|_{L^1}+C\|G_{HS}(t)\|_{L^1}\| V_0\|_{L^2}\nonumber\\
&&+C\int_0^t \|G_L(t - s)\|_{L^2} \|N(V)(s)\|_{L^1} ds\nonumber\\
 &&+C\int_0^t \|G_{HR}(t - s)\|_{L^2} \|N(V)(s)\|_{L^1} ds+C\int_0^t \|G_{HS}(t - s)\|_{L^1} \|N(V)(s)\|_{L^2} ds\nonumber\\
 &\leq&C(1+t)^{-\frac{3}{4}}+C\int_0^t\Big((1+t-s)^{-\frac{3}{4}}+e^{-C(t-s)}\Big)\|N(V)(s)\|_{L^1}ds\nonumber\\
 &&+C\int_0^t e^{-C(t-s)}\|N(V)(s)\|_{L^2}ds\nonumber\\
 &\leq&C(1+t)^{-\frac{3}{4}}+C\int_0^t\Big((1+t-s)^{-\frac{3}{4}}+e^{-C(t-s)}\Big)(1+s)^{-\alpha_1}ds\\
 &&+C\int_0^t e^{-C(t-s)}((1+s)^{-\alpha_2}ds\nonumber\\
 &\leq&C(1+t)^{-\frac{3}{4}},\nonumber
\end{eqnarray}
where $\alpha_1>1$ and $\alpha_2>1$ are verified in Section \ref{Bootstrap}.

We assume that equation \eqref{Hk decay} holds for $|\alpha|=k-1$, and inductively prove its validity for $|\alpha|=k, 1\leq k\leq 2.$
According to Duhamel's principle, we have
\begin{equation}
	\label{linear 7}
	\begin{split}
 \|\nabla^k V(t)\|_{L^2} \leq &\|\nabla^k(G(t) * V_0) \|_{L^2}+ \int_0^t \|\nabla^k(G(t - s) * N(V)(s))\|_{L^2} ds\\
 \leq&\|\nabla^k(G(t) * V_0) \|_{L^2}+\int_0^t \|\nabla^k(G_L(t - s) * N(V)(s))\|_{L^2} ds\\
    &+\int_0^t \|\nabla^k(G_{HR}(t - s) * N(V)(s))\|_{L^2} ds\\
    &+\int_0^t \|\nabla^k(G_{HS}(t - s) * N(V)(s))\|_{L^2} ds\\
 =:&\sum_{i=1}^{4}L_j.
 \end{split}
\end{equation}
Now we'd like to estimate them one by one. For $L_1$, we can derive it from \eqref{initial V},
\begin{equation}
	\begin{split}\label{initial}
		\norm{\nabla^k G (t)* V_0 }_{L^2}\le & \norm{\nabla^k G_L (t)* V_0 }_{L^2}+\norm{\nabla^k G_{HR} (t)* V_0 }_{L^2}+\norm{\nabla^k G_{HS} (t)* V_0 }_{L^2}\\
		\le &C\norm{\nabla^k G_L(t) }_{L^2}\norm{V_0}_{L^1}+C\norm{\nabla G_{HR}(t) }_{L^2}\norm{\nabla^{k-1} V_0}_{L^2}+Ce^{-Ct}\norm{\nabla^k V_0 }_{L^2}\\
		\le &C  (1+t)^{-\frac{3}{4}-\frac{k}{2}} + C e^{-Ct}+C e^{-Ct}\le C(1+t)^{-\frac{3}{4}-\frac{k}{2}}.
	\end{split}
\end{equation}
The $L_2$ can be splited into the following two term:
\begin{equation}\begin{aligned}
L_2\leq&\int_{0}^{\frac{t}{2}}\|\nabla^k G_{L}(t-s)\|_{L^2}\|N(V)(s)\|_{L^1}ds+\int_{\frac{t}{2}}^{t}\|\nabla G_{L}(t-s)\|_{L^1}\|\nabla^{k-1}N(V)(s)\|_{L^2}ds\\
=:&L_{21}+L_{22}.
\end{aligned}\end{equation}
For $L_{21}$, we have
\begin{equation}\begin{aligned}
L_{21}\leq&C\int_{0}^{\frac{t}{2}}\|\nabla^kG_{L}(t-s)\|_{L^2}\|N(V)(s)\|_{L^1}ds\\
\leq&C\int_{0}^{\frac{t}{2}}(1+t-s)^{-\frac{3}{4}-\frac{k}{2}}(1+s)^{-\alpha_1}ds\\
\leq&C(1+t)^{-\frac{3}{4}-\frac{k}{2}},
\end{aligned}\end{equation}
where $\alpha_1>1$ is ensured by Section \ref{Bootstrap}.
The term $L_{22}$ requires careful analysis and will be discussed in detail  using \eqref{V infty}, \eqref{grad V 2}, \eqref{grad2 V 2}, \eqref{nabla2v}, \eqref{gradient V infty}:
\begin{eqnarray}
L_{22}&\leq& C\int_{\frac{t}{2}}^{t}\|\nabla G_{L}(t-s)\|_{L^1}\||\nabla^{k-1}N_{\varrho}|+|\nabla^{k-1} N_1|+|\nabla^{k-1} N_2|+|\nabla^{k-1} N_3|)(s)\|_{L^2}ds\nonumber\\
&&+C\int_{\frac{t}{2}}^{t}\|\nabla G_{L}(t-s)\|_{L^{\frac{4}{3}}}\|\nabla^{k-1}N_4(s)\|_{L^{\frac{4}{3}}}ds\nonumber\\
&\leq& C\int_{\frac{t}{2}}^{t}(1+t-s)^{-\frac{1}{2}}\Big(\|V\|_{L^\infty}\|\nabla^{k}V\|_{L^2}\\
&&+\|\nabla V\|_{L^\infty}\|\nabla^{k-1}V\|_{L^2}+\|\nabla V\|_{L^\infty}\|\nabla^{k}V\|_{L^2}+\|\nabla V\|^2_{L^\infty}\|\nabla^{k-1}V\|_{L^2}\Big)(s)ds\nonumber\\
&&+\int_{\frac{t}{2}}^{t}(1+t-s)^{-\frac{7}{8}}\Big(\|\varrho\|_{L^\infty}\|\nabla^{k+1}u\|_{L^{\frac{4}{3}}}+\|\nabla^{k-1}\varrho\|_{L^2}\|\nabla^2u\|_{L^4}\Big)ds\nonumber\\
&\leq& C(1+t)^{-\frac{3}{4}-\frac{k}{2}},\nonumber
\end{eqnarray}
where we have used 
\begin{equation}\begin{aligned}
\|\nabla^3u\|_{L^{\frac{4}{3}}}\leq&C(1+t)^{-\frac{21}{20}}.
\end{aligned}\end{equation}
Regarding high-frequency nonlinear estimates, we obtain
\begin{equation}\begin{aligned}
L_3\leq&C\int_{0}^{t}\|\nabla^k(G_{HR}(t-s)*N(V)(s))\|_{L^2}ds\\
\leq&C\int_{0}^{t}\|\nabla G_{HR}(t-s)\|_{L^1}\|\nabla^{k-1}N(V)(s)\|_{L^2}ds\\
\leq&C\int_{0}^{t}e^{-C(t-s)}\Big(\|V\|_{L^\infty}\|\nabla^{k}V\|_{L^2}+\|\nabla V\|_{L^\infty}\|\nabla^{k-1}V\|_{L^2}+\|\nabla V\|_{L^\infty}\|\nabla^{k}V\|_{L^2}\\
&+\|\varrho\|_{L^\infty}\|\nabla^{k+1}u\|_{L^2}+\|\nabla^2u\|_{L^\infty}\|\nabla^{k-1}\varrho\|_{L^2}\Big)(s)ds\\
\leq& C(1+t)^{-\frac{3}{4}-\frac{k}{2}},
    \end{aligned}\end{equation}
where we have used \eqref{nabla2infty},\eqref{grad V 2}, \eqref{grad2 V 2}, \eqref{nabla2v} and
\begin{equation}\begin{aligned}\label{nabla3v}
\|\nabla^3V\|_{L^2}\leq& C\|\nabla^4V\|^{\frac{2}{5}}_{L^2}\|V\|^{\frac{3}{5}}_{L^\infty} \leq C(1+t)^{-\frac{9}{10}}.
\end{aligned}\end{equation}
    We establish the following nonlinear bound for the singular part
\begin{equation}\begin{aligned}\label{L44}
L_4\leq&C\int_{0}^{t}\|G_{HS}(t-s)*\nabla^k\text{div}(\varrho u)(s)\|_{L^2}ds\\
\leq&C\int_{0}^{t}e^{-C(t-s)}\|\nabla^k\text{div}(\varrho u)(s)\|_{L^2}ds\\
\leq&C\int_{0}^{t}e^{-c(t-s)}\Big((\|u\|_{L^\infty}+\|\varrho\|_{L^\infty})(\|\nabla^{k+1}\varrho\|_{L^2}+\|\nabla^{k+1}u\|_{L^2})\\
&+(\|\nabla u\|_{L^\infty}+\|\nabla\varrho\|_{L^\infty})(\|\nabla^{k}\varrho\|_{L^2}+\|\nabla^{k}u\|_{L^2})\Big)(s)ds\\
\leq&C\int_{0}^{t}e^{-C(t-s)}\Big(\|V\|_{L^\infty}\|\nabla^{k+1}V\|_{L^2}+\|\nabla V\|_{L^\infty}\|\nabla^{k}V\|_{L^2}\Big)(s)ds\\
\leq& C(1+t)^{-\frac{3}{4}-\frac{k}{2}},
    \end{aligned}\end{equation}
where we have used \eqref{grad V 2}, \eqref{grad2 V 2}, \eqref{nabla3v}, \eqref{V infty}, \eqref{gradient V infty}.

Putting together \eqref{initial} -\eqref{L44} yields
\begin{equation}\begin{aligned}
    \|\nabla^k V\|_{L^2}\leq C(1+t)^{-\frac{3}{4}-\frac{k}{2}}.
    \end{aligned}\end{equation}
Thus, we have completed the proof of Theorem \ref{main theorem}.

    \begin{rmk}
   When handling the singular terms, we need to distribute all derivatives to the nonlinear terms. As a result, we can only obtain the optimal decay for $\|\nabla^k u\|_{L^2}$ in this setting.     
    \end{rmk}

\section*{Data availability}
No data was used for the research described in the article.

\section*{Ackonwledgments}
X.-D. Huang is partially supported by CAS Project for Young Scientists in Basic
Research, Grant No. YSBR-031 and NNSFC Grant Nos. 12494542 and 11688101.

\end{document}